\documentclass[12pt]{article}

\usepackage[utf8]{inputenc} 
\usepackage[T1]{fontenc}
\usepackage[top=2cm, bottom=2cm, left=2cm, right=2cm]{geometry}
\usepackage{enumerate}
\usepackage{amsthm}
\usepackage{amsmath}
\usepackage{amssymb}
\usepackage{mathrsfs}
\usepackage{graphicx}
\usepackage{wrapfig}
\usepackage{makeidx}
\usepackage{dsfont}
\usepackage{accents}
\usepackage{bm}
\usepackage{hyperref}
\usepackage{tikz,ulem}
\usetikzlibrary{trees,positioning,decorations.pathreplacing}

\theoremstyle{plain}
\newtheorem{defi}{Definition}

\newtheorem{thm}{Theorem}[section]
\newtheorem{lem}{Lemma}[section]

\newtheorem{rem}{Remark}[section]

\newtheorem{theoA}{Theorem}

\newtheorem{propA}[theoA]{Proposition}

\newtheorem{lemA}[theoA]{Lemma}

\renewcommand{\l}{\left}
\renewcommand{\r}{\right}

\renewcommand{\o}[1]{\overline{#1}}
\renewcommand{\epsilon}{\varepsilon}

\newcommand{\Id}{\mathrm{Id}}
\newcommand{\Tr}{\mathrm{Tr}}

\makeatletter
\DeclareRobustCommand{\p}[1]{%
  \mathpalette\do@cev{#1}%
}
\newcommand{\do@cev}[2]{%
  \fix@cev{#1}{+}%
  \reflectbox{\(\m@th#1\vec{\reflectbox{\)\fix@cev{#1}{-}\m@th#1#2\fix@cev{#1}{+}\(}}\)}%
  \fix@cev{#1}{-}%
}
\newcommand{\fix@cev}[2]{%
  \ifx#1\displaystyle
    \mkern#23mu
  \else
    \ifx#1\textstyle
      \mkern#23mu
    \else
      \ifx#1\scriptstyle
        \mkern#22mu
      \else
        \mkern#22mu
      \fi
    \fi
  \fi
}

\makeatother

\newcommand{\R}{\mathbb{R}}

\newcommand{\calF}{\mathcal{F}}

\newcommand{\G}{\mathcal{G}}

\newcommand{\calE}{\mathcal{E}}

\newcommand{\E}{\mathbb{E}}
\renewcommand{\P}{\mathbb{P}}
\newcommand{\ind}{\mathds{1}}

\renewcommand{\em}{}

\title{A multi-dimensional version of Lamperti's relation and the Matsumoto-Yor processes}
\author{Thomas \textsc{Gerard}, Valentin \textsc{Rapenne}, Christophe \textsc{Sabot} and Xiaolin \textsc{Zeng}}

\begin{document}

\maketitle

\begin{abstract}
{The distribution of a one-dimensional drifted Brownian motion conditioned on its first hitting time to 0 is the same as a three-dimensional Bessel bridge. By applying the time change in Lamperti's relation to this result, Matsumoto and Yor \cite{MY01} showed a relation between Brownian motions with opposite drifts. In two subsequent papers \cite{matyor3,matyor2}, they established a geometric lifting of the process 2M-B in Pitman's theorem, known as the Matsumoto--Yor processs. They also established an equality in law involving Inverse Gaussian distribution and its reciprocal (as processes), known as the Matsumoto--Yor property, by conditioning some exponential Wiener functional.

In \cite{SabZenEDS}, Sabot and Zeng generalized the result on drifted Brownian motion conditioned on its first hitting. More precisely, they introduced a family of Brownian semimartingales with interacting drifts, for which when conditioned on the vector \(\tau\) (the hitting times to \(0\) of each component), their joint law is the same as for independent three-dimensional Bessel bridges. The distribution of \(\tau\) is a generalization of Inverse Gaussian distribution in multi-dimension and it is related to a random potential \(\beta\) that appears in the study of the Vertex Reinforced Jump Process.

The aim of this paper is to generalize some results in \cite{MY01,matyor3,matyor2} in the context of interacting Brownian martingales. We apply a Lamperti-type time change to the previous family of interacting Brownian motions and we obtain a multi-dimensional opposite drift theorem. Moreover, we also give a multi-dimensional counterpart of the Matsumoto--Yor process and it's intertwining relation with interacting geometric Brownian motions.}
\end{abstract}

\section{Introduction}
We start with a discussion on several classical properties of one-dimensional Brownian motion, for which we provide their multi-dimensional counterparts. Let \(\theta>0\) and \(\eta\geq 0\) be fixed and let \((B(t))_{t\geq 0}\) be a standard \(1\)-dimensional Brownian motion. Further, let us consider the drifted brownian motion
\[(X(t))_{t\geq 0}=(\theta +B(t)-\eta t)_{t\geq 0}.\]
Then, by \cite{W74,V91}, the distribution of the first hitting time \(\tau\) of \(0\) by \(X\) has a density which is known explicitly and conditionally on \(\tau\), \((X(t))_{0\leq t\leq \tau}\) is a \(3\)-dimensional Bessel bridge. We will give more details in Proposition \ref{prop-1d-bmhit}.

Another important result is the Matsumoto-Yor opposite drift theorem~\cite{MY01}. This theorem concerns a Brownian motion with negative drift \(-\mu\) which, when conditioned on some exponential functional of its sample path, can be represented as a Brownian motion with opposite drift \(\mu\), with an additional explicit corrective term. A version of this result is stated in Theorem \ref{thm:opp.dr}, in the special case where \(\mu=\frac{1}{2}\). One proof of the theorem relies on applying Lamperti's relation to the classic result on hitting times of Brownian motion. Lamperti's relation, presented in Proposition \ref{prop:lamperti}, provides a way to write any Bessel process with index \(\mu\) as the exponential of a time-changed Brownian motion with drift \(\mu\).

In \cite{matyor1,matyor3,matyor2}, Matsumoto and Yor proved some fundamental results concerning exponential functional of the Brownian motion. More precisely, if \(B\) is a standard Brownian motion and \(\mu\in\R\), let us define for every \(t>0\),
\[e_t^{(\mu)}=\exp(B_t+\mu t), \hspace{0.1 cm }A_t^{(\mu )}=\int_0^t(e_s^{(\mu)})^2ds  \text{ and }Z_t^{(\mu)}=A_t^{(\mu)}/e_t^{(\mu)}.\]
Then, the conditional distribution of \(e_t^{(\mu)}\) knowing \(Z_t^{(\mu)}=z\) is known explicitely and is a generalized Inverse Gaussian distribution.  Moreover, \(Z\) is a diffusion with respect to a filtration which is strictly smaller than the filtration of \(B\). These results imply a property of intertwinnings between the processes \(e\) and \(Z\).  While somehow mysterious at first sight, the Matsumoto-Yor properties have deep generalizations in the context of Lie groups and solvable polymers (see e.g. \cite{oconnel-Yor,oconnel,BBOI,BBOII,Chhaibi,Bourgade_review}), the classical Matsumoto-Yor properties being related to $sl_2(\R)$. In this paper we present a generalization of Matsumoto-Yor properties of a different type, based on lattice interactions rather than on general Lie groups.

In \cite{SabZenEDS}, Sabot and Zeng gave a multivariate version of the drifted Brownian motion \(X\) presented above. This multivariate version concerns a family of Brownian motions indexed by a finite graph with interacting drifts, defined as the solution of a system of stochastic differential equations (SDEs). They proved that the inverse of the hitting times of \(0\) for this family have the law of a random potential, which we denote \(\beta\), introduced by Sabot, Tarrès and Zeng in \cite{SabTarZen} and generalized in \cite{Let19}. That random potential is closely related to the supersymmetric hyperbolic sigma model studied by Disertori, Spencer and Zirnbauer in \cite{DisSpeZir} and \cite{Disertori2010}, and is central in the analysis of the Vertex Reinforced Jump Process in \cite{SabTar}, \cite{SabTarZen} and \cite{SabZen}. (See also \cite{Lupu2019a,Bauerschmidt2019b,Bauerschmidt2019,Collevecchio2018,Chang2019,Merkl2019a} for related discussions in statistical mechanics and random operators.) This paper and \cite{SabZenEDS} show that by many aspects, the distribution of the hitting times of these drifted interacting Brownian motions can be interpreted as a multivariate version of the Inverse Gaussian distribution.

The goal of this paper is two-fold:
\begin{enumerate}
\item Our first goal is to obtain a multi-dimensional version of the opposite drift theorem, by applying an analogue of Lamperti's time change to the family of interacting Brownian motions given in Theorem \ref{thm:SDE.X}. Difficulties arise in applying a time change to the interaction term, since the time change is different on each coordinate of the process. We can overcome this problem in two different ways : either by using the mixture representation given in Theorem \ref{thm:SDE.X}, and applying the time change to each independent 3-dimensional Bessel bridge ; or by using a form of strong Markov property verified by these interacting Brownian motions (c.f. Theorem \ref{thm:Markov.X} or Theorem 2 in \cite{SabZenEDS}).
\item
Another goal of this paper is to prove a multidimensional version of the conditional Matsumoto--Yor property for \(\mu=-\frac{1}{2}\). The proof of this property uses the multidimensional version of the opposite drift theorem. We also provide some intertwinnings and identities in law which generalize one-dimensional results of Matsumoto and Yor in \cite{matyor1}, \cite{matyor2} and \cite{matsumoto2003interpretation}. 
\end{enumerate}
\paragraph{Organisation of the paper:}  Section \ref{sec:opp.dr.lamp} recall some results for one-dimensional Brownian motion. In section \ref{subsec:recalls}, we list previous results of Sabot, Zeng and Tarrès on the random potential \(\beta\) and its realization as the limit of interacting Brownian motions. In Section \ref{subsec:multid}, we give multi-dimensional counterparts of the results of section \ref{sec:opp.dr.lamp}.  Two open questions are discussed in Section \ref{subsec:twoo}. The remaining sections are devoted to the proofs.
\section{Context and statement of the results}
\label{sec2}
\subsection{Results in dimension one}\label{sec:opp.dr.lamp}
We first recall a classical result about hitting time of drifted Brownian motion.
\begin{propA}
\label{prop-1d-bmhit}
Let \(\theta>0\) and \(\eta\geq 0\) be fixed, and let \(B=(B(t))_{t\ge 0}\) be a standard one-dimensional Brownian motion. We define the  Brownian motion \(X=(X(t))_{t\ge 0}\) with drift \(\eta\) by
\[X(t)=\theta + B(t) - \eta t,\ t\ge 0.\]
If \(\tau\) is the first hitting time of \(0\) by \(X\), i.e.
\begin{equation}
\label{equation-T0}
\tau=\inf\{t\ge 0,\ B(t)+\theta-\eta t=0\},
\end{equation}
then the distribution of \(\tau\) is given by  an Inverse Gaussian distribution \(IG( \frac{\theta}{\eta},\theta^2)\), its density is
\begin{equation}\label{T.dist}
\frac{\theta}{\sqrt{2\pi t^3}}\exp\l( -\frac{1}{2}  \frac{(\theta - \eta t)^2}{t}  \r)\ind_{\{t\geq 0\}}dt.
\end{equation}
Moreover, conditionally on \(\tau\), \(\big( X(t) \big)_{0\leq t\leq \tau}\) has the same distribution as a \(3\)-dimensional Bessel bridge from \(\theta\) to \(0\) on the time interval \([0,\tau]\).
\end{propA}
The first part is a classical easy computation, the last statement is a consequence of Williams' decomposition, see \cite{W74}~Theorem 3.1 or \cite{RY13}~p. 255, Theorem (3.11). In particular, when \(\eta=0\), \textit{i.e.} \(X\) is a Brownian motion without drift, then \eqref{T.dist} is the density of inverse Gamma, i.e. it has the distribution of \(\frac{1}{2\gamma}\), where \(\gamma\) is a Gamma random variable with parameter \((\frac{1}{2},\theta^2)\). 

Now, let us present a version of the Matsumoto-Yor opposite drift theorem from \cite{MY01}, in the specific case where the drift \(\mu\) is \(\frac{1}{2}\), with an additional boundary term depending on \(\eta\).

\begin{theoA}\label{thm:opp.dr}[Theorem 2.2 and Proposition 3.1 in \cite{MY01}]
Let \(\theta>0\) and \(\eta\geq 0\) be fixed, and let \(B\) be a standard one-dimensional Brownian motion. We define the process \(\rho\) as the solution of the following SDE :
\begin{equation}\label{eq-rho-u}
\begin{aligned}
\rho(u)= \log(\theta) + B(u) - \frac{1}{2}u -\int_0^u \eta e^{\rho(v)}dv
\end{aligned}
\end{equation}
for \(u\geq 0\). Define \(T(u)=\int_0^u e^{2\rho(v)}dv\).
\begin{itemize}
\item[(i)] We have 
\[
T(u)\xrightarrow[u\to\infty]{a.s.} T^{\infty},
\]
where \(T^{\infty}\) has the same distribution as \(\tau\) in Proposition \ref{prop-1d-bmhit}: it is distributed according to 
\[
\frac{\theta}{\sqrt{2\pi t^3}}\exp\l( -\frac{1}{2}  \frac{(\theta - \eta t)^2}{t}  \r)\ind_{\{t\geq 0\}}dt.
\]
\item[(ii)] Conditionally on \(T^{\infty}\), there exists a standard one-dimensional Brownian motion \(B^*\) such that for \(u\geq 0\),
\begin{equation}\label{eq-rho-u-cond}
\begin{aligned}
\rho(u)=\log(\theta)+B^*(u)+\frac{1}{2}u+\log\left( \frac{T^{\infty}-T(u)}{T^{\infty}} \right) .
\end{aligned}
\end{equation}
\end{itemize}
\end{theoA}
Note that Theorem \ref{thm:opp.dr} holds for any \(\mu >0\), except that the law of \(T^{\infty}\) needs to be adjusted accordingly.
One proof of Theorem \ref{thm:opp.dr} relies on applying a time change to the  result on hitting times of the Brownian motion introduced in Proposition \ref{prop-1d-bmhit}. We sketch this proof now. The relevant time change is the one that appears in Lamperti's relation, presented below (see e.g. \cite{RY13} p.452) :

\begin{propA}\label{prop:lamperti}[Lamperti's relation]
Let \((\rho(u))_{u\geq 0}\) be a drifted Brownian motion with drift \(\mu\in\R\). For \(u \geq 0\), define
\[
T(u)=\int_0^u \exp(2\rho(v))dv.
\]
Then there exists a Bessel process \((X(t))_{t\geq 0}\) with index \(\mu\), starting from \(1\), such that for \(u\geq 0\),
\[
e^{\rho(u)}=X(T(u)).
\]
\end{propA}

To prove Theorem \ref{thm:opp.dr} using this time change,  we use the same notations as in the introduction. Fix \(\theta>0\), \(\eta>0\), and let \(X(t)=\theta+ B(t) +\eta t\) where \(B\) is a standard Brownian motion. Let \(U(t)=\int_0^t \frac{1}{X(s)^2}ds\), and denote \(T=U^{-1}\). Note that \(T\) is the analogue of the time change featured in Lamperti's relation, where \(X\) plays the role of a drifted Bessel process with index \(-\frac{1}{2}\).

If \(\rho(u)=\log(X(T(u)))\), then by Itô formula, \(\rho\) has the distribution as the solution of Equation \eqref{eq-rho-u}. Moreover, when \(u\to\infty\), by definition \(T(u)\to \tau\), where \(\tau\) is the first hitting time of \(0\) by \(X\). By Proposition \ref{prop-1d-bmhit}, conditionally on \(\tau\), \(X\) has the distribution of a \(3\)-dimensional Bessel bridge, \textit{i.e.} a Bessel bridge with index \(\frac{1}{2}\), using this fact in applying Itô formula, we found that the conditional law of \(\rho\)  equals the law of the solution to Equation \eqref{eq-rho-u-cond}.

\begin{rem}
In \cite{MY01}, the opposite drift theorem is stated in a different form, where \(\eta=0\) and the drift \(\mu\) can be different from \(\frac{1}{2}\). Its proof still relies on applying Lamperti's relation, but this time to a result concerning hitting times of Bessel processes with any index \(-\mu\) (see \cite{L18} and \cite{PY81}).
\end{rem}

Now, let us recall the conditional Matsumoto-Yor property  \cite{matyor1,matyor3,matyor2}. For every \(\alpha\in\R\), \(K_{\alpha}\) is the modified Bessel function of the second kind with index \(\alpha\).
\begin{theoA}\label{thm:matyorcond}
Let \(\mu\in\R\). Let \(B\) be a Brownian motion. For every \(t>0\), define
\[e_t^{(\mu)}=\exp(B_t+\mu t), \hspace{0.1 cm }A_t^{(\mu )}=\int_0^t(e_s^{(\mu)})^2ds \text{ and }Z_t^{(\mu)}=A_t^{(\mu)}/e_t^{(\mu)}.\]
Let \((\mathcal{Z}_t^{(\mu)})_{t\geq 0})\) be the natural filtration associated with the process \((Z_t^{(\mu)})_{t\geq 0}\).
\begin{enumerate}[(i)]
\item For every \(t>0\), \[\mathcal{Z}_t^{(\mu)}\subsetneqq\sigma(B_s,s\leq t).\]
\item  There exists a Brownian motion \(\tilde{B}\) such that \((Z_t^{(\mu)})_{t\geq 0}\) is solution of the SDE
\[dZ_t^{(\mu)}=Z_t^{(\mu)}d\tilde{B}_t+\left(\frac{1}{2}-\mu\right)Z_t^{(\mu)}dt+\frac{K_{1+\mu}}{K_{\mu}}\left(\frac{1}{Z_t^{(\mu)}} \right)dt.\]
\item \label{iii-MY} For any \(t>0\),
\[\P(e_t^{(\mu)}\in dx|\mathcal{Z}_t^{(\mu)},Z_t^{(\mu)}=z)=\frac{x^{\mu-1}}{2K_{\mu}(1/z)}\exp\left(-\frac{1}{2z}\left( \frac{1}{x}+x\right) \right)dx.\]
\end{enumerate}
\end{theoA}
Note that the conditional law in \ref{iii-MY} is a generalized Inverse Gaussian distribution. Moreover, in \cite{matyor2}, Matsumoto and Yor showed that Theorem \ref{thm:matyorcond} implies the following property of intertwinnings.
\begin{theoA}\label{thm:intertwinning1d}
Let \(\mu\in\R\). Let \((\tilde{P}^{(\mu)}_t)_{t\geq 0}\) be the semigroup of \(e^{(\mu)}\) and let \((\tilde{Q}^{(\mu)}_t)_{t\geq 0}\) be the semigroup of the diffusion \(Z^{(\mu)}\).
Then, for every \(t\geq 0\),
\[\tilde{Q}_t\tilde{K}^{(\mu)}=\tilde{K}^{(\mu)}\tilde{P}_t\]
where for every measurable function \(g\) from \(\R\) into \(\R_+\),
\[\tilde{K}^{(\mu)}(g)(z)=\int_{0}^{+\infty}g(x)\frac{x^{\mu-1}}{2K_{\mu}(1/z)}\exp\left(-\frac{1}{2z}\left( \frac{1}{x}+x\right) \right)dx.\]
\end{theoA}
The aim of this article is now to obtain a multi-dimensional version of Theorem \ref{thm:opp.dr} , \ref{thm:matyorcond} and \ref{thm:intertwinning1d}. In particular, we are able to apply the time change from Lamperti's relation to the forthcoming Theorem \ref{thm:SDE.X}.
\subsection{Brownian motions with interacting drifts and the random \(\beta\) potential}
\label{subsec:recalls}
Let \(\G=(V,E)\) be a finite, connected, and non-oriented graph, endowed with conductances \((W_e)_{e\in E}\in(\R_+^*)^E\). For \(i,j\in V\), we denote by \(W_{i,j}=W_{\{i,j\}}\) if \(\{i,j\}\in E\), otherwise \(W_{i,j}=0\). Note that it is possible to have \(W_{i,i}>0\). For \(\beta\in\R^V\), we define \(H_\beta=2\beta-W\), where \(W\) is the graph adjacency matrix \(W=(W_{i,j})_{i,j\in V}\), and \(\beta\) denotes here abusively the diagonal matrix with diagonal coefficients \((\beta_i)_{i\in V}\); in particular, \(H_\beta\) is a \(V\times V\) matrix.

{\bf Notation:} For a vector \(x=(x_i)_{i\in V}\in \mathbb{R}^V\), we sometimes simply write \(x\) for the diagonal matrix with diagonal coefficients \(x_i\), there will be no ambiguity thanks to the context.
\begin{propA}[Theorem 4 in \cite{SabTarZen}, Theorem 2.2 in \cite{Let19}]
\label{prop:nuwtheta}
For all \(\theta\in(\R_+^*)^V\) and \(\eta\in(\R_+)^V\), the measure \(\nu_V^{W,\theta,\eta}\) defined by
\[\nu_V^{W,\theta,\eta}(d\beta)=\ind_{H_\beta>0}\l(\frac{2}{\pi}\r)^{|V|/2}\exp\l(-\frac{1}{2}\langle\theta,H_\beta\theta\rangle-\frac{1}{2}\langle\eta,(H_\beta)^{-1}\eta\rangle+\langle\eta,\theta\rangle\r)\frac{\prod_{i\in V}\theta_i}{\sqrt{|H_\beta|}}d\beta\]
is a probability distribution. Moreover, for all \(i\in V\), the random variable \(\frac{1}{2\beta_i-W_{i,i}}\) has Inverse Gaussian distribution with parameter \((\frac{\theta_i}{\eta_i+\sum_{j\neq i}W_{i,j}\theta_j},\theta_i^2)\). Furthermore, for every \((\lambda_i)\in\R_+^V\), 
\[\int\exp\left(-\sum\limits_{i\in V}\lambda_i\beta_i\right)d\nu^{W,\theta,\eta}(\beta)=e^{-\frac{1}{2}\langle\sqrt{\theta^2+\lambda},W\sqrt{\theta^2+\lambda}\rangle+\frac{1}{2}\langle\theta,W\theta\rangle+\langle\eta,\theta-\sqrt{\theta^2+\lambda}\rangle}\times \prod\limits_{i\in V}\frac{\theta_i}{\sqrt{\theta_i^2+\lambda_i}} .\]
\end{propA}

For \(t\in(\R_+)^V\), we also denote by \(K_t\) the matrix \(K_t=\Id-tW\), where \(t\) still denotes the diagonal matrix with coefficients \((t_i)_{i\in V}\). Note that if \(t\in(\R_+^*)^V\), we have \(K_t=t H_{\frac{1}{2t}}\), where \(\frac{1}{2t}=\l(\frac{1}{2t_i}\r)_{i\in V}\). Finally, for \(t\in(\R_+)^V\) and \(T\in(\o{\R_+})^V\), we define the vector \(t\wedge T=(t_i\wedge T_i)_{i\in V}\). 

\begin{theoA}\label{thm:SDE.X}[Lemma 1 and Theorem 1 in \cite{SabZenEDS}]
Let \(\theta\in(\R_+^*)^V\) and \(\eta\in(\R_+)^V\) be fixed, and let \((B_i(t))_{i\in V,t\geq 0}\) be a standard \(|V|\)-dimensional Brownian motion.
\begin{itemize}
\item[(i)] The following stochastic differential equation (SDE) has a unique pathwise solution :
\begin{equation}\label{SDE:X}
X_i(t)=\theta_i+\int_0^t\ind_{s<\tau_i}dB_i(s) -\int_0^t\ind_{s<\tau_i}((W\psi)(s)+\eta)_i ds \tag{\(E_V^{W,\theta,\eta}(X)\)}
\end{equation}
for \(i\in V\) and \(t\geq 0\), where for \(i\in V\), \(\tau_i\) is the first hitting time of \(0\) by \(X_i\), and for \(t\geq 0\),
\[
\psi(t)=K_{t\wedge \tau}^{-1}(X(t)+(t\wedge \tau)\eta).
\]

\item[(ii)] If \((X_i)_{i\in V}\) is solution of \eqref{SDE:X}, the vector \(\l(\frac{1}{2 \tau_i}\r)_{i\in V}\) has distribution \(\nu_V^{W,\theta,\eta}\), and conditionally on \((\tau_i)_{i\in V}\), the paths \((X_i(t))_{0\leq t\leq \tau_i}\) are independent \(3\)-dimensional Bessel bridges.
\end{itemize}
\end{theoA}

To obtain an analogue of the opposite drift theorem as in Section \ref{sec:opp.dr.lamp}, we want to apply the time change from Lamperti's relation to solutions \((X_i)_{i\in V}\) of \eqref{SDE:X}. A problem will arise in the interaction term, since the time change will be different on every coordinate of \(X\). To solve this, we will use a form of strong Markov property verified by solutions of \eqref{SDE:X}, which is a consequence of Theorem \ref{thm:SDE.X}(ii). This Markov property will be true with respect to multi-stopping times, defined as follows. 

\begin{defi}\label{defi:multistop}
Let \(X\) be a multi-dimensional càdlag process indexed by \(V\). A random vector \(T=(T_i)_{i\in V} \in \overline{\R_+}^V\) is called a multi-stopping time with respect to \(X\) if for all \(t\in\R_+^V\), the event \(\cap_{i\in V}\{T_i\leq t_i\}\) is \(\mathcal{F}^X_t\)-measurable, where 
\[
\mathcal{F}^X_t=\sigma \big( (X_i(s))_{0\leq s\leq t_i}, i\in V \big).
\]
In this case, we denote by \(\mathcal{F}^X_T\) the \(\sigma\)-algebra of events anterior to \(T\), \textit{i.e.}
\[
\mathcal{F}^X_T=\l\{ A\in\mathcal{F}^X_\infty, \forall t\in\R_+^V, A\cap\{T_i\leq t_i\}\in\mathcal{F}^X_{t} \r\}
\]
\end{defi}

Let us now formulate the strong Markov property for solutions of \eqref{SDE:X}.

\begin{theoA}\label{thm:Markov.X}[Theorem 2 (iv) in \cite{SabZenEDS}]
Let \(X\) be a solution of \eqref{SDE:X}, and \(T=(T_i)_{i\in V}\) be a multi-stopping time with respect to \(X\). Define the shifted process \(Y\) by
\[
Y_i(t)=X_i(T_i+t)
\]
for \(i\in V\) and \(t\geq 0\). Moreover, we denote
\[
\widetilde{W}^{(T)}=W\l(K_{T\wedge \tau}\r)^{-1},\widetilde{\eta}^{(T)}=\eta+\widetilde{W}^{(T)}\big((T\wedge \tau)\eta\big), \text{ and } X(T)=(X_i(T_i))_{i\in V}.
\]
On the event \(\cap_{i\in V}\{T_i<\infty\}\), conditionally on \(T\) and \(\mathcal{F}^X_T\), the process \(Y\) has the same distribution as the solution of \(\big(E_V^{\widetilde{W}^{(T)},X(T),\widetilde{\eta}^{(T)}}(X)\big)\).
\end{theoA}

\subsection{Main results : A multi-dimensional version of the opposite drift theorem, the conditional Matsumoto-Yor property and consequences}
\label{subsec:multid}
Let \((X_i)_{i\in V}\) be a solution of \eqref{SDE:X}. As in the usual case of Lamperti's relation, let us introduce the functional that will define the time change. For \(i\in V\) and \(t\geq 0\), we set 
\[
U_i(t)= \int_0^t \frac{\ind_{s<\tau_i}}{X_i(s)^2}ds.
\]
It turns out that for any \(i\in V\), \(\lim_{t\to \tau_i}U_i(t)=+\infty\) a.s. It will be proved in Lemma \ref{lem:lim.u}. Therefore, for all \(i\in V\), we can define \(T_i=(U_i|_{[0,\tau_i[})^{-1}\). In particular, for all \(u\geq 0\), \(T_i(u)<\tau_i\). Moreover \(\lim_{u\to\infty}T_i(u)=\tau_i\). Thus, in this time scale, it is natural to prefer the notation \(T_i^{\infty}:=\tau_i\) for every \(i\in V\).

We will show that the time-changed solution \(\big( X_i\circ T_i \big)_{i\in V}\) can be written as
\[
X_i(T_i(u))=e^{\rho_i(u)}
\]
for \(u\geq 0\), where \((\rho_i)_{i\in V}\) is solution of a new sytem of stochastic differential equations :

\begin{thm}\label{thm:SDE.rho}
\begin{itemize}
\item[(i)]
For \(i\in V\) and \(u\geq 0\), let us define \(\rho_i(u)=\log \big( X_i(T_i(u)) \big)\). Then \((\rho,T)\) is solution of the following system of SDEs :
\begin{equation}\label{SDE:rho}
\l\{\begin{aligned}
	\rho_i(v) &= \log(\theta_i) + \widetilde{B}_i(v) + \int_0^v \l(-\frac{1}{2}-e^{\rho_i(u)}\l(\widetilde{W}^{(u)}(e^{\rho(u)}+T(u)\eta)+\eta\r)_i\r)du,\\
	T_i(v) &= \int_0^v e^{2\rho_i(u)}du,
\end{aligned} \r. \tag{\(E_V^{W,\theta,\eta}(\rho)\)}
\end{equation}
for \(i\in V\) and \(v\geq 0\), where \((\widetilde{B}_i)_{i\in V}\) is a \(|V|\)-dimensional standard Brownian motion, \(e^{\rho(u)}\) denotes the vector \((e^{\rho_i(u)})_{i\in V}\), and 
\[
\widetilde{W}^{(u)}=W K_{T(u)}^{-1}=W \big( \Id-T(u)W \big)^{-1}.
\]
\item[(ii)] The equation \eqref{SDE:rho} admits a unique pathwise solution \(u\mapsto \big( \rho(u),T(u) \big)\), which is a.s. well defined on all of \(\R_+\).
\end{itemize}
\end{thm}

As a consequence of Theorems \ref{thm:SDE.X}(ii) and \ref{thm:SDE.rho}, we can  relate the solutions of \eqref{SDE:rho} to time-changed Bessel bridges and the distribution \(\nu_V^{W,\theta,\eta}\). This is stated in Theorem \ref{thm:md.opp.dr} below, which is the multi-dimensional version of Theorem \ref{thm:opp.dr}.

\begin{thm}\label{thm:md.opp.dr}
Let \((\rho,T)\) be solution of \eqref{SDE:rho}.
\begin{itemize}
\item[(i)]
For all \(i\in V\), we have
\[
T_i(u)=\int_0^u e^{2\rho_i(v)}dv \xrightarrow[u\to\infty]{a.s.}T_i^{\infty},
\]
where \(\l(\frac{1}{2T_i^{\infty}}\r)_{i\in V}\) is distributed according to \(\nu_V^{W,\theta,\eta}\). \item[(ii)]
There exists a standard \(|V|\)-dimensional Brownian motion \(B^*\) which is independent of \(T^{\infty}\) such that for \(i\in V\) and \(u\geq 0\),
\[
\rho_i(u) =\log(\theta_i)+B^*_i(u)+\frac{1}{2}u+\log \left( \frac{T_i^{\infty}-T_i(u)}{T_i^{\infty}} \right).
\]
In particular, the processes \((\rho_i,T_i)_{i\in V}\) are independent conditionally on \(T^{\infty}\). 
\end{itemize}
\end{thm}
The multidimensional counterpart of Theorem \ref{thm:matyorcond} is as follows.
\begin{thm}\label{thm:md.matyorcond}
Let \((\rho,T)\) be a solution of \(E_V^{W,\theta,\eta}(\rho)\). For every \(u\geq0\) and for every \(i\in V\), we define \(Z_i(u)=T_i(u)\exp(-\rho_i(u))\). Let us denote by \((\mathcal{Z}_u)_{u\geq 0}\) the natural filtration associated with \((Z(u))_{u\geq 0}\). Then, it holds that
\begin{enumerate}[(i)]
\item For every \(u>0\), \[\mathcal{Z}_u\subsetneqq\sigma(\rho_v,v\leq u).\]
\item The process \((Z(u))_{u\geq 0}\) has independent components and for every \(i\in V\), there exists a Brownian motion \(\widehat{B}_i\) such that \((Z_i(u))_{u\geq 0}\) is solution of the SDE
\[dZ_i(u)=Z_i(u)d\widehat{B}_i(u)+(\theta_i+Z_i(u))du.\]
\item \((Z(u))_{u\geq 0}\) is independent of \(T^{\infty}\).
\item \label{iv-MY} For every \(u>0\), let us define \(\beta(u)=(\beta_i(u))_{i\in V}=(1/(2T_i(u))_{i\in V}\). Then, for every \(u>0\), the conditional law of \(\beta(u)\) given \(\mathcal{Z}_u\) is
\[\mathcal{L}\left(\beta(u)\Bigg|\mathcal{Z}_u,Z(u)=z\right)=\nu_V^{W,\theta,\eta+1/z}\]
where \(\eta+1/z\) is the vector \((\eta_i+1/z_i)_{i\in V}\).

\end{enumerate}
We emphasize the fact that correlations coming from interacting drifts is contained in the conditional law in \ref{iv-MY}.
\end{thm}

If \(\rho \) be a solution of \(E_V^{W,\theta,\eta}(\rho)\), for every \(u\geq0\) and for every \(i\in V\), we define \(Z_i(u)=T_i(u)\exp(-\rho_i(u))\).
An important consequence of the previous result is that \((\rho,T)\) and \(Z\) are related via Markov intertwinings which generalize Theorem \ref{thm:intertwinning1d}.

More precisely, by Theorem \ref{thm:md.matyorcond}, \((Z(u))_{u\geq 0}\) is a Markov process in its own sigma-field \((\mathcal{Z}_u)_{u\geq 0}\). Let us denote its semigroup by \((Q_u)_{u\geq 0}\). Moreover, by Theorem \ref{thm:SDE.rho}, \((\rho,T)\) is also a Markov process. Let us denote its semigroup by \((P_u)_{u\geq 0}\).
\begin{thm}\label{thm:intertwinnings}Let \(\theta\in(\R_+^*)^V \) and \(\eta\in(\R_+)^V\). Let \((\rho,T)\) be a solution of \(E_V^{W,\theta,\eta}(\rho))\) and let \(Z=Te^{-\rho}\).  Then, \((\rho,T)\) and \(Z\) are intertwinned in the following sense: for every \(u\in\R_+\),
\[Q_{u}\circ K=K\circ P_{u}\]
where for every measurable function \(g\) from \((\R^V)^2\) into \(\R_+\),
\[K(g)(z)= \int g\left(\Big(-\ln({2\beta_i z_i})\Big)_{i\in V},\left(\frac{1}{2\beta_i} \right)_{i\in V} \right)d\nu_V^{W,\theta,\eta+1/z}(\beta). \]
\end{thm}
The kernel \(K\) in Theorem \ref{thm:intertwinnings} comes from (iv) in Theorem \ref{thm:md.matyorcond}.
Moreover, thanks to Theorem \ref{thm:md.matyorcond}, we are able to give some new identities involving the measure \(\nu_V^{W,\theta,\eta}\). These identities look like other identities in law which are known in the \(1\)-dimensional case. More details are given in the next section.
\begin{thm}\label{thm:equalities_in_law}
Let \(\theta\in(\R_+^*)^V \), \(\eta\in\R_+^V\) and \(z\in (\R_+^*)^V\). On the one hand, let \((\beta_i)_{i\in V}\) be a random potential with distribution \(\nu_V^{W,\theta,\eta+1/z}\). Conditionally on \(\beta\), let \((\alpha_i)_{i\in V}\) be a random potential with distribution \(\nu_V^{\widetilde{W},\widetilde{\theta},\widetilde{\eta}}\) where \(\widetilde{W}=WK_{1/(2\beta)}^{-1}\), \(\widetilde{\eta}=\eta+WH_{\beta}^{-1}\eta\) and for every \(i\in V\), \(\widetilde{\theta}_i=1/(2\beta_i z_i)\). On the other hand, let \((\delta_i)_{i\in V}\) be a random potential with distribution \(\nu_V^{W,\theta,\eta}\). Let \((A_i)_{i\in V}\) be a family of independent random variables such that for every \(i\in V\), \(A_i\) is distributed like \(IG(1/(\theta_iz_i),1/z_i^2)\). Moreover we assume that \((A_i)_{i\in V}\) is independent of \(\delta\). Then it holds that,

\[\left(\left(2\beta_i\right)_{i\in V}, \left(\frac{(2\beta_i)^2}{2\alpha_i}\right)_{i\in V}\right)\overset{law}=\left(\left(2\delta_i+A_i\right)_{i\in V},\left(A_i+\frac{A_i^2}{2\delta_i} \right)_{i\in V} \right).\]
\end{thm}
\subsection{Two open questions}
\label{subsec:twoo}
\subsubsection*{The Matsumoto-Yor property}

The Gamma and Inverse Gaussian distributions, as well as the inverse Gamma and reciprocal Inverse Gaussian distributions, all fall into the family of the so-called generalized Inverse Gaussian distributions.

A random variable is said to have generalized Inverse Gaussian distribution with parameter \((q, a,b)\) where \(q\in \mathbb{R}\) and \(a,b>0\), and denoted \(\operatorname{GIG}(q, a,b)\) if it has the following density:
\begin{equation}
\label{equation-gig-density}
\left( \frac{a}{b} \right)^{q/2} \frac{1}{2K_q(\sqrt{ab})} t^{q-1} e^{-\frac{1}{2}(at+b/t)}\mathds{1}_{t>0}.
\end{equation}
In particular, we have the following special cases (where zero parameter is understood as limit, see e.g. ~\cite{norman1994johnson}):
\[
\operatorname{IG}\l(\frac{\theta}{\eta}, \theta^2\r)=\operatorname{GIG}\l(-\frac{1}{2}, \frac{\eta^2}{2}, \frac{\theta^2}{2}\r),\ \operatorname{Gamma}\l(\frac{1}{2} ,\theta^2\r)=GIG\l(\frac{1}{2}, 0, \frac{\theta^2}{2}\r)
\]
and
\[
X\sim \operatorname{GIG}\l(-\frac{1}{2},\frac{\eta^2}{2},\frac{\theta^2}{2}\r) \Leftrightarrow 1/X\sim \operatorname{GIG}\l(\frac{1}{2}, \frac{\theta^2}{2}, \frac{\eta^2}{2}\r).
\]
Define the last visit of 0 of our drifted Brownian motion to be \(\widetilde{\tau}=\sup\{t\ge 0:\ B_t+\theta-\eta t=0\}\).
By a time inversion argument, i.e. setting
\[\widetilde{B}_t=\begin{cases}-tB_{1/t} & t>0 \\ 0 & t=0\end{cases},\]
the Gaussian process \(\widetilde{B}\) is also a Brownian motion and we deduce that \({\widetilde{\tau}}^{-1}\) is the first visit time to 0 of \(\widetilde{B}_t+\eta-\theta t\), hence \(\widetilde{\tau}\) is \(\operatorname{GIG}(\frac{1}{2}, \frac{\eta^2}{2},\frac{\theta^2}{2})\) distributed, moreover, \(\widetilde{\tau}-\tau\) is \(\operatorname{Gamma}(\frac{1}{2}, \theta^2)\) distributed and by Strong Markov property of Brownian motion, it is independent of \(\tau\). There exists also an interpretation of the generalized Inverse Gausian distribution with  any index \(q\) as stopping time of some diffusion process \cite{BNBH,V91} but we focus on the case \(q=\pm1/2\) in this paper.

More generally, we have the following identity in distribution, which is known as the Matsumoto-Yor property~\cite{Wesolowski2007,matsumoto2003interpretation}:
\begin{propA}
	\label{my-prop}
	Let \((\tau,\widetilde{\tau})\) be a random vector, then there is equivalence between the following statements:
	\begin{itemize}
		\item[(i)] \(\left(\displaystyle \frac{1}{\tau},\widetilde{\tau}-\tau  \right)\sim \operatorname{GIG}(\frac{1}{2},\frac{\theta^2}{2},\frac{\eta^2}{2})\otimes \operatorname{Gamma}(\frac{1}{2}, \theta^2)\)
		\item[(ii)] \(\left(\displaystyle \frac{1}{\tau}-\frac{1}{\widetilde{\tau}}, \widetilde{\tau} \right)\sim \operatorname{Gamma}(\frac{1}{2}, \eta^2)\otimes \operatorname{GIG}(\frac{1}{2}, \frac{\eta^2}{2},\frac{\theta^2}{2})\).
	\end{itemize}
\end{propA}
It is not clear how we can get a multi-dimensional generalization of Proposition \ref{my-prop}. However, the identity of Theorem \ref{thm:equalities_in_law} can be rewritten as
\[\left(\frac{1}{A_i+2\delta_i},\frac{1}{A_i}-\frac{1}{A_i+2\delta_i} \right)_{i\in V}\overset{law}=\left(\frac{1}{2\beta_i},\frac{2\alpha_i}{(2\beta_i)^2}\right)_{i\in V}.\]
This identity can be viewed as a weaker form of (i) in Proposition \ref{my-prop} where we lost the independence property.
Recall that for every \(i\in V\), \(1/(2\beta_i)\) is the first hitting time of zero of \(X_i\) (that is, a drifted Brownian motion with interaction). Moreover, \(1/A_i\) is distributed like \(GIG\left(1/2,\frac{1}{2z_i^2},\frac{\theta_i^2}{2}\right)\), that is, like the last hitting time of zero by a Brownian motion with drift \(1/z_i\) without interaction. Perhaps the identity of Theorem \ref{thm:equalities_in_law} can be interpreted in term of a coupling between interacting and non-interacting Brownian motions. 

Properties in Proposition~\ref{my-prop} has been called the Matsumoto--Yor (MY) property by Stirzaker \cite{Stirzaker2005} p. 43. Letac and Wesolowski \cite{LetacWe2000} provided a characterization theorem related to MY property, namely, if \(\tau,\widetilde{\tau}\) are random variables s.t. \(\frac{1}{\tau},\widetilde{\tau}-\tau\) are independent and \(\frac{1}{\tau}-\frac{1}{\widetilde{\tau}},\widetilde{\tau}\) are independent, then they necessarily follow the law prescribed in Proposition~\ref{my-prop}. It is tempting to say that a multi-dimensional counterpart of such characterization law also holds.
\subsubsection*{An opposite-drift theorem for other values of the drift}

The multi-dimensional opposite-drift theorem proved in this paper is limited to the case of the drift \(-\frac{1}{2}\), since it results from Theorem \ref{thm:SDE.X}, which concerns Bessel processes with index \(-\frac{1}{2}\) and \(\frac{1}{2}\) (\textit{i.e.} Brownian motion and \(3\)-dimensional Bessel bridges). We could try to obtain a similar result for other values of the drift \(\mu\). This necessitates the use of a random potential analogous to \(\beta\), whose marginals would relate to the hitting times of Bessel processes with other indices, that is, generalized Inverse Gaussian distributions. A natural candidate for the distribution of the potential associated with the drift \(-\mu\) with \(\mu>0\) is the measure \(\nu_{V,\mu}^{W,\eta}\) with density:
\[\nu_{V,\mu}^{W,\theta,\eta}(d\beta)=C(\mu,W,\eta)\ind_{H_\beta>0}\l(\frac{2}{\pi}\r)^{|V|/2}\exp\l(-\frac{1}{2}\langle\theta,H_\beta\theta\rangle-\frac{1}{2}\langle\eta,(H_\beta)^{-1}\eta\rangle\r)\frac{\prod_{i\in V}\theta_i}{|H_\beta|^{\mu-1}}d\beta\] 
where \(C(\mu,W,\eta)\) is a normalizing constant.
Nonetheless, the explicit density of \(\beta\) plays an important role in our proof. When \(\mu\ne \frac{1}{2}\), its normalizing constant \(C(\mu,W,\eta)\) is no longer a constant, but depends on the underlying graph and our proof does not apply directly to such cases.

The case of index \(\mu=\frac{3}{2}\) might be solvable thanks to recent developments by Bauerschmidt, Crawford, Helmuth and Swan in \cite{BCHS19}, and by Crawford in \cite{Cra19}. These articles concern other sigma models, in particular \(\mathbb{H}^{2|4}\), which is related to random spanning forests, The normalizing  constant \(C(3/2,W,\eta)\) of \(\beta\) associated to this model is the partition function of random forests. It might be a candidate if one looks for a generalization of the \(\beta\) potential corresponding to index \(\frac{3}{2}\). Moreover, the SDE given by \((i)\) in Theorem \ref{thm:SDE.rho} should be much more complicated.

\section{Multi-dimensional time change : Proof of Theorem \ref{thm:SDE.rho} and Theorem \ref{thm:md.opp.dr}}
\subsection{Justification of the Lamperti time change}
Recall that
\[
U_i(t)= \int_0^t \frac{\ind_{s<\tau_i}}{X_i(s)^2}ds
\]
where \((X_i)_{i\in V}\) is a solution of \eqref{SDE:X}.
\begin{lem}\label{lem:lim.u}
For any \(i\in V\),
\[
\lim_{t\to \tau_i} U_i(t)=+\infty,
\]
consequently \(U_i:[0,\tau_i[\,\to[0,+\infty[\) is a.s. a bijection.
\end{lem}
\begin{proof}
Let \(X=(X_i)_{i\in V}\) be a solution of \eqref{SDE:X}. According to Theorem \ref{thm:SDE.X}, conditionally on \((\tau_i)_{i\in V}\), the trajectories \((X_i(t))_{0\leq t\leq \tau_i}\) are independent three-dimensional Bessel bridges. As a consequence, to prove Lemma \ref{lem:lim.u}, it suffices to show the same result for a three-dimensional Bessel bridge.

Fix \(\theta>0\) and \(\tau>0\), let \(X\) be a three-dimensional Bessel bridge from \(\theta\) to \(0\)  in the time interval \([0,\tau]\). 
Since \(X\) is a three-dimensional Bessel bridge, there exists a standard Brownian motion \(B\) such that
\[
dX(t)=dB(t)+\frac{1}{X(t)}dt-\frac{X(t)}{\tau-t}dt,
\]
therefore by Ito's lemma, for \(t<\tau\),
\begin{align*}
d\log(X(t)) &= \frac{dB(t)}{X(t)}+\frac{dt}{X(t)^2}-\frac{dt}{\tau-t} -\frac{1}{2}\frac{dt}{X(t)^2} \\
&= dM(t) +\frac{1}{2}dU(t) +d\log\l(\tau-t\r),
\end{align*}
where \(M(t)=\int_0^t \frac{dB(t)}{X(t)}\) is a martingale, and \(\langle M \rangle_t=U(t)\). Therefore, there exists a standard Brownian motion \(\widehat{B}\) such that \(M(t)=\widehat{B}(U(t))\). Thus, for \(t\geq 0\), we have
\begin{equation}\label{tc.Bb}
\log(X(t))=\log(\theta)+\widehat{B}(U(t))+\frac{1}{2}U(t)+\log\l(\frac{\tau-t}{\tau}\r),
\end{equation}
\textit{i.e.}
\[
\frac{X(t)}{\tau-t}=\frac{\theta}{\tau}e^{\widehat{B}(U(t))+\frac{1}{2}U(t)}.
\]
On the other hand, since \(X\) is a three-dimensional Bessel bridge, there exists (see \cite{RY13} p.467) a three-dimensional Bessel process \(Y\) such that for \(t\geq 0\),
\[
X(t)=(\tau-t)Y\l(\frac{t}{\tau(\tau-t)}\r).
\]
Therefore, when \(t\to \tau\), we have a.s. \(\frac{X(t)}{\tau-t}\to+\infty\). Since \(u\mapsto\widehat{B}(u)+\frac{1}{2}u\) cannot explode in finite time, we have necessarily 
\[
U(t)\xrightarrow[t\to \tau]{a.s.}+\infty.
\]
\end{proof}

\subsection{Proof of Theorem \ref{thm:SDE.rho}}
\begin{proof}[Proof of Theorem \ref{thm:SDE.rho} (ii)]
Assume that Theorem \ref{thm:SDE.rho} (i) is proven, \textit{i.e.} that \eqref{SDE:rho} has almost surely a unique pathwise solution defined on all of \(\R_+\). Let \(\widetilde{B}\) be a \(|V|\)-dimensional Brownian motion. Thanks to Theorem \ref{thm:SDE.rho} (i), we know that \eqref{SDE:rho} admits a solution that is well defined on \(\R_+\). Let us now show that this solution is necessarily unique.

Let \((\rho^*,T^*)\) be another solution of \eqref{SDE:rho} with the Brownian motion \(\widetilde{B}\). Let also \(\mathcal{K}\) be a compact subset of \(\R^V\times \{t\in\R_+^V, K_t>0\}\) containing \((\log(\theta_i),0)_{i\in V}\). Then the function
\[
\begin{cases}
\mathcal{K}  \to \R^V\times\R^V \\
(\rho,t)  \mapsto \bigg( -\frac{1}{2} - e^{\rho_i} \Big( W K_t^{-1} (e^\rho+t\eta) +\eta \Big)_i \, , \, e^{2\rho_i} \bigg)_{i\in V}
\end{cases}
\]
is bounded and Lipschitz. Therefore, up to the stopping time \(U_\mathcal{K}=\inf\{u\geq 0, (\rho(u),T(u))\notin\mathcal{K}\}\), we have \((\rho(u),T(u))=(\rho^*(u),T^*(u))\) from Theorem 2.1, p.375 of \cite{RY13}. Since this is true for all compact subset \(\mathcal{K}\) of \(\R^V\times \{t\in\R_+^V, K_t>0\}\), we have a.s. \((\rho,T)=(\rho^*,T^*)\). This concludes the proof of Theorem \ref{thm:SDE.rho} (ii).
\end{proof}

Let's prove Theorem \ref{thm:SDE.rho} (i). Let \(B\) be a standard \(|V|\)-dimensional Brownian motion, and let \((X_i)_{i\in V}\) be a solution of \eqref{SDE:X}. For \(i\in V\), recall that \(T_i\) is the inverse function of
\[
U_i: \begin{cases}
\l[0,T^{\infty}_i\r[ \to  \l[0,+\infty\r[\\
t  \mapsto \int_0^t \frac{ds}{X_i(s)^2}
\end{cases}
\]
and \(\rho_i(u)=\log\big(X_i(T_i(u))\big)\) for \(u\geq 0\). In order to show that \((\rho,T)\) is solution of \eqref{SDE:rho}, we want to apply the same time change as in Lamperti's relation. However, in the equation \eqref{SDE:X}, the interactive drifts provided by \(\psi(t)\) to each coordinates \(X_i(t)\) are calculated at the same time \(t\geq 0\), while in Lamperti's time scale, i.e. for \(X_i(t)=e^{\rho_i(U_i(t))}\), the times \(U_i(t)\) are different at each coordinates \(i\in V\).

We present here two different ways to overcome this problem. The first proof relies on identifying the infinitesimal generator of the process \((\rho_i,T_i)_{i\in V}\), using the strong Markov property presented in Theorem \ref{thm:Markov.X}. The second one uses Theorem \ref{thm:SDE.X} (ii), i.e. \(X\) is a mixture of independent Bessel bridges, to which we can apply the time change separately, and then identify the law of the annealed process using Girsanov's theorem.

\subsection{First proof of Theorem \ref{thm:SDE.rho} (i) : using the strong Markov property of Theorem \ref{thm:Markov.X}}
\begin{proof}[Proof of Theorem \ref{thm:SDE.rho} (i)]
Firstly, let \(u\geq 0\) be fixed, and \(f:\R^V\times\R^V\to \R\) be a compactly supported \(\mathcal{C}^2\) function. To identify the infinitesimal generator of \((\rho,T)\), let us compute
\[
\lim_{v\to u^+}\frac{\E[f(\rho(v),T(v))|\mathcal{F}^{(\rho,T)}_u]-f(\rho(u),T(u))}{v-u}.
\]
Note that \((T_i(u))_{i\in V}\) is a multi-stopping time in the sense of Definition \ref{defi:multistop} and that \(\mathcal{F}^{(\rho,T)}_u=\mathcal{F}^X_{T(u)}\). Define
\[
\widetilde{W}^{(u)}= W \big( K_{T(u)} \big)^{-1}, \; \widetilde{K}^{(u)}_t = \Id-t\widetilde{W}^{(u)}, \text{ and } \widetilde{\eta}^{(u)}= \eta + \widetilde{W}^{(u)}(T(u)\eta).
\]
Thanks to Theorem \ref{thm:Markov.X}, conditionally on \(\mathcal{F}^X_{T(u)}\), the shifted process 
\[
Y=Y^{(u)}:t\mapsto \big( X_i(T_i(u)+t) \big)_{i\in V}
\]
is the solution of
\[
\begin{cases}
dY_i(t) = \ind_{t\leq\widehat{T}^0_i}d\widehat{B}_i(t) -\ind_{t\leq\widehat{T}^0_i} \l( \widetilde{W}^{(u)} (\widetilde{K}^{(u)}_t)^{-1} \l(Y(t)+(t\wedge \widehat{T}^0)\widetilde{\eta}^{(u)}\r) +\widetilde{\eta}^{(u)} \r)_i dt, & i\in V, t\ge 0\\
Y_i(0) = X_i(T_i(u)), & i\in V
\end{cases}
\]
where \(\widehat{B}\) is a \(|V|\)-dimensional standard Brownian motion, independent of \(\mathcal{F}^X_{T(u)}\), and \(\widehat{T}^0_i\) is the first hitting time of \(0\) by \(Y_i\). 

Fix \(v>u\), define the interrupted process
\[
Z=Z^{(u,v)}:t\mapsto \Big( X_i \big( (T_i(u)+t)\wedge T_i(v) \big)  \Big)_{i\in V}.
\]
For all \(i\in V\) and \(t\geq 0\), we have \(Z_i(t)=Y_i\l(t\wedge \widehat{T}_i(v) \r)\), where \(\widehat{T}_i(v)=T_i(v)-T_i(u)\). Therefore, \(Z\) is the solution of 
\[
\begin{cases}
dZ_i(t) = \ind_{t\leq \widehat{T}_i(v)}d\widehat{B}_i(t) -\ind_{t\leq \widehat{T}_i(v)}\l( \widetilde{W}^{(u)} (\widetilde{K}^{(u)}_t)^{-1} \l(Y(t)+(t\wedge \widehat{T}^0)\widetilde{\eta}^{(u)}\r) +\widetilde{\eta}^{(u)} \r)_i dt, & i\in V,\ t\ge 0\\
Z_i(0) = X_i(T_i(u)), & i\in V.
\end{cases}
\]
Moreover, since \(\widehat{T}_i(v)<\widehat{T}_i^0<\infty\) a.s. for all \(i\in V\), there exists a.s. \(\widehat{T}^\infty\) large enough so that \(Z_i(t)=Y_i(\widehat{T}_i(v))=X_i(T_i(v))\) for all \(i\in V\) and \(t\geq\widehat{T}^\infty\).

By Ito's lemma, for all \(t\geq 0\) we have
\begin{align*}
d\log(Z_i(t)) =&\, \ind_{t\leq \widehat{T}_i(v)} \frac{d\widehat{B}_i(t)}{Z_i(t)} - \ind_{t\leq \widehat{T}_i(v)}\frac{dt}{2 Z_i(t)^2}\\
&  -\ind_{t\leq \widehat{T}_i(v)}\l( \widetilde{W}^{(u)} (\widetilde{K}^{(u)}_t)^{-1} \l(Y(t)+(t\wedge \widehat{T}^0)\widetilde{\eta}^{(u)}\r) +\widetilde{\eta}^{(u)} \r)_i \frac{dt}{Z_i(t)},
\end{align*}
where we can replace \(t\wedge \widehat{T}^0\) with \(t\), as \(\widehat{T}_i(v)<\widehat{T}^0\). For \(i\in V\), let \(\widehat{M}_i\) be the martingale \(\widehat{M}_i(t)=\int_0^t \frac{d\widehat{B}_i(s)}{Z_i(s)}\) for \(t\geq 0\).

If we denote 
\[
\Phi(t)= \big(\log(Z_i(t)), (T_i(u)+t)\wedge T_i(v)\big)_{i\in V}\in\R^V\times\R^V
\]
for \(t\geq 0\), then, applying Ito's lemma to \(t\mapsto f(\Phi(t))\), we get
\begin{align*}
f(\Phi(t)) &- f(\Phi(0)) = \sum_{i\in V} \int_0^{t\wedge  \widehat{T}_i(v)} \frac{\partial f}{\partial \rho_i}(\Phi(s)) d\widehat{M}_i(s) + \sum_{i\in V} \int_0^{t\wedge  \widehat{T}_i(v)} \frac{1}{2}\frac{\partial^2 f}{\partial\rho_i^2}(\Phi(s))\frac{ds }{Z_i(s)^2} \\
&+ \sum_{i\in V} \int_0^{t\wedge \widehat{T}_i(v)} \frac{\partial f}{\partial \rho_i}(\Phi(s)) \l(-\frac{1}{2}- Z_i(s)\l( \widetilde{W}^{(u)} (\widetilde{K}^{(u)}_s)^{-1} \big(Y(s)+s\widetilde{\eta}^{(u)}\big) +\widetilde{\eta}^{(u)} \r)_i\r) \frac{ds}{Z_i(s)^2} \\
&+ \sum_{i\in V} \int_0^{t\wedge  \widehat{T}_i(v)}\frac{\partial f}{\partial t_i}(\Phi(s))ds.
\end{align*}
Taking \(t\geq\widehat{T}^\infty\), we get \(t\wedge \widehat{T}_i(v)=\widehat{T}_i(v)\) for all \(i\in V\), and 
\[
f(\Phi(t))-f(\Phi(0))=f(\rho(v),T(v))-f(\rho(u),T(u)),
\]
since \(\rho_i(w)=\log\l(X_i(T_i(w))\r)\) for \(w\in\R_+\) and \(i\in V\). For all \(i\in V\), we can now use the following time change in the corresponding integrals above : \(s=T_i(w)-T_i(u)=\widehat{T}_i(w)\), \textit{i.e.} \(w=U_i(T_i(u)+s)\). Note that for \(0\leq s \leq \widehat{T}_i(v)\),
\begin{align*}
\frac{d}{ds}U_i(T_i(u)+s) =\frac{1}{X_i(T_i(u)+s)^2}=\frac{1}{Z_i(s)^2},
\end{align*}
and for \(u\leq w \leq v\),
\[\frac{d}{dw}\widehat{T}_i(w) =X_i(T_i(w))^2=e^{2\rho_i(w)}.\]
Thus,
\begin{align*}
f\big(\rho(v),&T(v)\big)- f\big(\rho(u),T(u)\big) \\
=&\, \sum_{i\in V} \bigg( \int_u^v \frac{\partial f}{\partial \rho_i}\big(\rho(w),T(w)\big) d\widehat{M}_i(\widehat{T}_i(w)) + \int_u^v \frac{1}{2}\frac{\partial^2 f}{\partial\rho_i^2}\big(\rho(w),T(w)\big)dw  \\
&+ \int_u^v \frac{\partial f}{\partial \rho_i}\big(\rho(w),T(w)\big) \l(-\frac{1}{2}- e^{\rho_i(w)}\l( \widetilde{W}^{(u)} (\widetilde{K}^{(u)}_{\widehat{T}_i(w)})^{-1} \big(X(T_i(w))+\widehat{T}_i(w)\widetilde{\eta}^{(u)}\big) +\widetilde{\eta}^{(u)} \r)_i\r) dw \\
&+ \int_u^v \frac{\partial f}{\partial t_i}\big(\rho(w),T(w)\big)e^{2\rho_i(w)}dw \bigg).
\end{align*}
Note that, the vector \(X(T_i(w))=\big( X_j(T_i(w)) \big)_{j\in V}\) is different from \(e^{\rho(w)}=\big( X_j(T_j(w)) \big)_{j\in V}\). This is why we need to take \(v\to u\) and identify the generator.

Since \(\widehat{B}\) is independent from  \(\mathcal{F}^{X}_{T(u)}\), we have
\[
\E \l[ \int_u^v \frac{\partial f}{\partial \rho_i}\big(\rho(w),T(w)\big) d\widehat{M}_i(\widehat{T}_i(w)) \middle| \mathcal{F}^{(\rho,T)}_u\r]=\E\l[ \int_0^{\widehat{T}_i(v)} \frac{\partial f}{\partial \rho_i}(\Phi(t)) \frac{d\widehat{B}_i(s)}{Z_i(s)}\middle| \mathcal{F}^{X}_{T(u)}\r]=0
\]
for all \(i\in V\), and therefore
\begin{flalign*}
  &\E \Big[ f\big(\rho(v),T(v)\big) -f\big(\rho(u),T(u)\big) \Big| \mathcal{F}^{(\rho,T)}_u \Big] &\ & \\
  =& \E\Bigg[\sum_{i\in V} \Bigg(  \int_u^v \frac{1}{2}\frac{\partial^2 f}{\partial\rho_i^2}\big(\rho(w),T(w)\big)dw& \\
  &\; + \int_u^v \frac{\partial f}{\partial \rho_i}\big(\rho(w),T(w)\big) \l(-\frac{1}{2}- e^{\rho_i(w)}\l( \widetilde{W}^{(u)} (\widetilde{K}^{(u)}_{\widehat{T}_i(w)})^{-1} \big(X(T_i(w))+\widehat{T}_i(w)\widetilde{\eta}^{(u)}\big) +\widetilde{\eta}^{(u)} \r)_i\r) dw &\\
  &\; + \int_u^v \frac{\partial f}{\partial t_i}\big(\rho(w),T(w)\big)e^{2\rho_i(w)}dw \Bigg)\Bigg|\rho(u),T(u)\Bigg].
\end{flalign*}
By continuity and dominated convergence, we conclude that
\begin{align*}
  &\lim_{v\to u^+} \frac{1}{v-u} {\E\l[ f\big(\rho(v),T(v)\big) \middle| \mathcal{F}^{(\rho,T)}_u\r]  -f\big(\rho(u),T(u)\big)}\\
  &= \sum_{i\in V}\Bigg( \frac{1}{2}\frac{\partial^2 f}{\partial\rho_i^2}\big(\rho(u),T(u)\big)\\
&\; \; + \frac{\partial f}{\partial \rho_i}\big(\rho(u),T(u)\big) \bigg( -\frac{1}{2}- e^{\rho_i(u)} \big( \widetilde{W}^{(u)} e^{\rho(u)}+\widetilde{\eta}^{(u)} \big)_i \bigg) + \frac{\partial f}{\partial t_i}\big(\rho(u),T(u)\big) e^{2\rho_i(u)}\Bigg),
\end{align*}
which is \(\mathcal{L}f(u)\), where \(\mathcal{L}\) is the infinitesimal generator associated with the system of SDEs \eqref{SDE:rho}.
\end{proof}
\subsection{Second proof of Theorem \ref{thm:SDE.rho} (i) : using the mixing measure and Girsanov's theorem}
This proof follows the same structure as that of Theorem \ref{thm:SDE.X}. We start from the distribution of the process as a mixture of simpler quenched processes, and we compute the integral in order to identify the annealed distribution, using Girsanov's theorem.
\begin{proof}[Alternative proof of Theorem \ref{thm:SDE.rho} (i)]
Let \(X=(X_i(t))_{i\in V, t\geq 0}\) be the canonical process in \(\mathcal{C}(\R_+,\R^V)\), and \(\P\) be the distribution on \(\mathcal{C}(\R_+,\R^V)\) under which \(X\) is solution of \eqref{SDE:X}. According to Theorem \ref{thm:SDE.X} (ii), the vector \((\beta_i)_{i\in V}=\l(\frac{1}{2\tau_i}\r)_{i\in V}=\l(\frac{1}{2T_i^{\infty}}\r)_{i\in V}\) has distribution \(\nu_V^{W,\theta,\eta}\). Moreover, conditionally on \((T_i^{\infty})_{i\in V}\), the marginal processes \(X_i\) for \(i\in V\) are independent \(3\)-dimensional Bessel bridges from \(\theta_i\) to \(0\) on \([0,T_i^{\infty}]\). In other words, we can write
\[
\P[\cdot]=\int \l(\bigotimes_{i\in V} \P_i^{\beta_i} [\cdot] \r) \nu_V^{W,\theta,\eta}(d\beta),
\]
where for \(i\in V\), \(\P_i^{\beta_i}\) is the distribution on \(\mathcal{C}(\R_+,\R)\) under which the canonical process \(X_i\) is a \(3\)-dimensional Bessel bridge from \(\theta_i\) to \(0\) on \([0,T_i^{\infty}]\).

Conditionally on \((\beta_i)_{i\in V}\), we can apply the time change independently to each marginal \(X_i\). According to the computations done in the proof of Lemma \ref{lem:lim.u}, In particular, Equation \eqref{tc.Bb}, we know that under \(\P_i^{\beta_i}\), there exists a standard Brownian motion \(\widehat{B}_i\) such that
\[
\rho_i(u) = \log(\theta_i) + \widehat{B}_i(u) + \frac{1}{2}u  + \log\l(\frac{T_i^{\infty}-T_i(u)}{T_i^{\infty}}\r), u\ge 0
\]
where \(T_i^{\infty}=\frac{1}{2\beta}\) and \(T_i(u)=\int_0^u e^{2\rho_i(v)}dv\), \textit{i.e.}
\[
\rho_i(u) = \log(\theta_i) + \widehat{B}_i(u) + \frac{1}{2}u  + \log\l(1-2\beta_i\int_0^u e^{2\rho_i(v)} dv \r).
\]
For each \(i\in V\), define a martingale \(L_i\) by
\[
L_i(u)= \int_0^u \l( -\frac{1}{2} + \frac{2\beta_i e^{2\rho_i(v)}}{1-2\beta_i\int_0^v e^{2\rho_i(s)}ds} \r) d\widehat{B}_i(v),\ u\ge 0.
\]
Clearly \(\rho_i(u)= \widehat{B}_i(u) - \langle\widehat{B}_i,L_i\rangle_u\). We can then introduce a probability distribution \(\widehat\P_i\) such that for all \(u\geq 0\),
\[
\E\l[\frac{d\widehat\P_i}{d\P_i^{\beta_i}}\middle|\calF^i_u\r]=\calE(L_i)(u),
\]
where \(\calF^i_u=\sigma\l(\widehat{B}_i(v),0\leq v\leq u\r)\), and \(\calE(L_i)(u)=e^{L_i(u)-\frac{1}{2}\langle L_i,L_i\rangle_u}\) is the exponential martingale associated with \(L_i\). By Girsanov's theorem, \(\rho_i\) is a standard Brownian motion under \(\widehat\P_i\). Note that \(\widehat\P_i\) does not depend on \(\beta_i\).

From now on, let us write \(\phi_i(u)=1-2\beta_i\int_0^u e^{2\rho_i(v)}dv=\frac{T_i^{\infty}-T_i(u)}{T_i^{\infty}}\) for \(u\geq 0\) and \(i\in V\). The following lemma gives an expression of \(\calE(L_i)\).

\begin{lem}\label{lemEiu}
For \(i\in V\) and \(u\geq 0\), define
\[
E_i(u)=\exp\l(-\theta_i^2\beta_i +\frac{\beta_i e^{2\rho_i(u)}}{\phi_i(u)} -\frac{1}{2}\rho_i(u)+\frac{1}{8}u\r) \phi_i(u)^{3/2}\sqrt{\theta_i}.
\]
We have \(\calE(L_i)=E_i\).
\end{lem}

\begin{proof}[Proof of Lemma \ref{lemEiu}]
Since \(E_i(0)=1\) almost surely, it suffices to show that \(\frac{dE_i(u)}{E_i(u)}=dL_i(u)\) for all \(u\geq 0\). Note that \(\rho_i(u) = \widehat{B}_i(u) + \frac{1}{2}u + \log (\phi_i(u))\), therefore
\[
E_i(u)=\exp\l(-\theta_i^2\beta_i +\beta_i\phi_i(u) e^{2\widehat{B}_i(u) + u} -\frac{1}{2}\widehat{B}_i(u)-\frac{1}{8}u\r) \phi_i(u)\sqrt{\theta_i}.
\]
By Ito's lemma, for \(u\geq 0\) we have
\begin{align*}
dE_i(u) &= \l(2\beta_i\phi_i(u)e^{2\widehat{B}_i(u)+u}-\frac{1}{2}\r)E_i(u) d\widehat{B}_i(u) \\
& \qquad + \frac{1}{2}\l(\l(2\beta_i\phi_i(u)e^{2\widehat{B}_i(u)+u}-\frac{1}{2}\r)^2+4\beta_i\phi_i(u)e^{2\widehat{B}_i(u)+u}\r)E_i(u) du \\
& \qquad + \l(\beta_i\l(\phi_i(u)+\phi_i'(u)\r)e^{2\widehat{B}_i(u)+u}-\frac{1}{8}+\frac{\phi_i'(u)}{\phi_i(u)}\r)E_i(u) du.
\end{align*}
Since \(\phi_i'(u)=-2\beta_i e^{2\rho_i(u)}=-2\beta_i\phi_i(u)^2 e^{2\widehat{B}_i(u)+u}\), we get
\begin{align*}
\frac{dE_i(u)}{E_i(u)} &= \l(-\frac{1}{2}+\frac{2\beta_i e^{2\rho_i(u)}}{\phi_i(u)}\r) d\widehat{B}_i(u) +\l( 2\beta_i^2\phi_i(u)^2 e^{4\widehat{B}_i(u)+2u} +\frac{1}{8}+\beta_i\phi_i(u)e^{2\widehat{B}_i(u)+u} \r. \\
&\l.\qquad + \beta_i\phi_i(u)e^{2\widehat{B}_i(u)+u} -2\beta_i^2\phi_i(u)^2 e^{4\widehat{B}_i(u)+2u} -\frac{1}{8} -2\beta_i\phi_i(u)e^{2\widehat{B}_i(u)+u}\r) du \\
&= dL_i(u).
\end{align*}

\end{proof}

Fix \(u\geq 0\), for any event \(A_u\in\calF_u^\rho=\sigma\l(\rho(v),0\leq v\leq u\r)\), we have
\begin{align*}
\P[A_u] & = \int \l(\bigotimes_{i\in V} \P_i^{\beta_i}[A_u]\r)\nu_V^{W,\theta,\eta} = \int \int_{A_u} \prod_{i\in V} \l(E_i(u)^{-1}d\widehat{\P}_i\r) \nu_V^{W,\theta,\eta}(d\beta) \\
& = \int_{A_u} D(u) d\widehat{\P}
\end{align*}
where \(\widehat{\P}=\bigotimes_{i\in V}\widehat{\P}_i\) and for \(u\geq 0\),
\[
D(u)=\int\l(\prod_{i\in V} E_i(u)^{-1}\r)\nu_V^{W,\theta,\eta}(d\beta).
\]
Let's compute \(D(u)\) now, and express it as an exponential martingale, so as we can apply Girsanov's theorem once again, and identify the distribution of \(\rho\) under \(\P\).

For \(u\geq 0\), we have
\begin{align*}
D(u)&=\int \exp\l( \sum_{i\in V}\l(\theta_i^2\beta_i-\frac{\beta_i e^{2\rho_i(u)}}{\phi_i(u)}+\frac{1}{2}\rho_i(u)-\frac{1}{8}u\r) \r)\frac{1}{\prod_{i\in V}\phi_i(u)^{3/2}\sqrt{\theta_i}} \\
		&\qquad \ind_{H_\beta>0}\l(\frac{2}{\pi}\r)^{|V|/2}\exp\l(-\frac{1}{2}\langle\theta,H_\beta\theta\rangle-\frac{1}{2}\langle\eta,(H_\beta)^{-1}\eta\rangle+\langle\eta,\theta\rangle\r)\frac{\prod_{i\in V}\theta_i d\beta_i}{\sqrt{|H_\beta|}}.
\end{align*}
In order to compute this integral, we will introduce a change of variables, and obtain an integral against the distribution \(\nu_V^{\widetilde{W}^{(u)}, \widetilde{\theta}^{(u)}, \widetilde\eta^{(u)}}\), where \(\widetilde{W}^{(u)}\), \(\widetilde{\theta}^{(u)}\) and \(\widetilde\eta^{(u)}\) are new parameters depending on the trajectory of \(\rho\) up to time \(u\) defined below.

Let us introduce the following notations : for \(u\geq 0\),
\[
\begin{cases}
\beta_i^{(u)} = \frac{1}{2T_i(u)} \text{ for } i\in V,\\
H^{(u)} = 2\beta^{(u)}-W, \\
K^{(u)} = T(u) H^{(u)} = \Id - T(u)W.
\end{cases}
\]
Now, we define the new following parameters
\[
\begin{cases}
\widetilde{W}^{(u)} = W (K^{(u)})^{-1} =W+W(H^{(u)})^{-1}W, \\
\widetilde\eta^{(u)} = \widetilde{W}^{(u)}T(u)\eta + \eta,\\
\widetilde\theta_i^{(u)} = e^{\rho_i(u)}\text{ for } i\in V
\end{cases}
\]
as well as the following associated quantities
\[
\begin{cases}
\widetilde{T}_i(u) = \frac{1}{2\beta_i}-T_i(u) =\frac{\phi_i(u)}{2\beta_i}\text{ for } i\in V, \\
\widetilde\beta_i^{(u)} = \frac{1}{2\widetilde{T}_i(u)} =\frac{\beta_i}{\phi_i(u)} \text{ for } i\in V,\\
\widetilde{H}^{(u)} = 2\widetilde\beta^{(u)}-\widetilde{W}^{(u)} ,\\
\widetilde{K}^{(u)} = \widetilde{T}(u) \widetilde{H}^{(u)} = \Id - \widetilde{T}(u)\widetilde{W}^{(u)}.
\end{cases} 
\]
Using these new notations, we can write
\begin{equation}\label{tech1}
\sum_{i\in V}\frac{\beta_i e^{2\rho_i(u)}}{\phi_i(u)}=\sum_{i\in V}(\widetilde\theta_i^{(u)})^2 \widetilde\beta_i^{(u)}=\frac{1}{2}\l\langle\widetilde\theta^{(u)}, \l(\widetilde{H}^{(u)}+\widetilde{W}^{(u)}\r) \widetilde\theta^{(u)}\r\rangle
\end{equation}
for \(u\geq 0\). The following technical lemma will allow us to express \(D(u)\) as an integral against \(\nu_V^{\widetilde{W}^{(u)}, \widetilde{\theta}^{(u)}, \widetilde\eta^{(u)}}\).

\begin{lemA}[Lemma 2 in \cite{SabZenEDS}]
\label{lem:tech}
\begin{enumerate}[(i)] For \(u\geq 0\), with the quantities defined just above, we have
\item \(K_{1/2\beta}=\widetilde{K}^{(u)}K^{(u)}\)
\item \(\widetilde\eta^{(u)} = T(u)^{-1}(H^{(u)})^{-1}\eta \)
\item \(\langle\widetilde\eta^{(u)},(\widetilde{H}^{(u)})^{-1}\widetilde\eta^{(u)}\rangle = \langle\eta,H_\beta^{-1}\eta\rangle - \langle\eta,(H^{(u)})^{-1}\eta\rangle.\)
\end{enumerate}
\end{lemA}
Using Lemma \ref{lem:tech} (i), we get that for \(u\geq 0\),
\[
H_\beta=2\beta K_{1/2\beta}=2\beta\widetilde{K}^{(u)}K^{(u)}=2\beta\widetilde{T}(u)\widetilde{H}^{(u)}K^{(u)},
\]
where \(2\beta_i\widetilde{T}_i(u)=1-\frac{\beta_i}{\beta_i^{(u)}}=\phi_i(u)\) for \(i\in V\). Therefore, we have
\begin{equation}\label{tech2}
\prod_{i\in V}\phi_i(u)^{3/2}\sqrt{|H_\beta|}=\prod_{i\in V}\phi_i(u)^2\sqrt{|\widetilde{H}^{(u)}|}\sqrt{|K^{(u)}|},
\end{equation}
where 
\[
\frac{d\widetilde\beta_i^{(u)}}{d\beta_i}=\frac{1}{\Big(1-\frac{\beta_i}{\beta_i^{(u)}}\Big)^2}=\frac{1}{\phi_i(u)^2}.
\]
Moreover, for all \(u\geq 0\) we have
\begin{equation}\label{tech3}
\ind_{H_\beta>0}=\ind_{H(u)>0}\ind_{\widetilde{H}(u)>0}.
\end{equation}
Combining equations \eqref{tech1}, \eqref{tech2} and \eqref{tech3}, as well as Lemma \ref{lem:tech} (iii), we finally obtain :
\begin{align*}
D(u) &= \l[ \int \ind_{\widetilde{H}^{(u)}>0}\l(\frac{2}{\pi}\r)^{|V|/2}\exp\l( -\frac{1}{2}\langle\widetilde{\theta}^{(u)},\widetilde{H}^{(u)}\widetilde{\theta}^{(u)}\rangle -\frac{1}{2}\langle\widetilde\eta^{(u)},(\widetilde{H}^{(u)})^{-1}\widetilde\eta^{(u)}\rangle +\langle\widetilde\eta^{(u)},\widetilde\theta^{(u)}\rangle\r) \r. \\
		&\qquad \l. \frac{\prod_{i\in V}\widetilde\theta_i^{(u)}}{\sqrt{|\widetilde{H}^{(u)}|}}\prod_{i\in V}\frac{d\widetilde\beta_i^{(u)}}{d\beta_i}d\beta_i\r] \ind_{H^{(u)}>0} \exp\l(\frac{1}{2}\langle\theta,W\theta\rangle+\langle\eta,\theta\rangle\r)\prod_{i\in V}\sqrt{\theta_i} \\
		&\qquad \exp\l(-\frac{1}{2}\langle\widetilde\theta^{(u)},\widetilde{W}^{(u)}\widetilde\theta^{(u)}\rangle -\frac{1}{2}\langle\eta,(H^{(u)})^{-1}\eta\rangle -\langle\widetilde\eta^{(u)},\widetilde\theta^{(u)}\rangle\r)\frac{\prod_{i\in V}\exp\l(\frac{1}{2}\rho_i(u)-\frac{1}{8}u\r)}{\sqrt{|K^{(u)}|}\prod_{i\in V}\widetilde\theta_i^{(u)}} \\
	&= \ind_{H^{(u)}>0} \exp\l(-\frac{1}{2}\langle\widetilde\theta^{(u)},\widetilde{W}^{(u)}\widetilde\theta^{(u)}\rangle -\frac{1}{2}\langle\eta,(H^{(u)})^{-1}\eta\rangle -\langle\widetilde\eta^{(u)},\widetilde\theta^{(u)}\rangle \r) \\
		&\qquad \frac{\prod_{i\in V}\exp\l(\frac{1}{2}\rho_i(u)-\frac{1}{8}u\r)}{\sqrt{|K^{(u)}|}}\exp\l(\frac{1}{2}\langle\theta,W\theta\rangle+\langle\eta,\theta\rangle\r)\prod_{i\in V}\sqrt{\theta_i},
\end{align*}
since the integral between brackets becomes
\[
\int \nu_V^{\widetilde{W}^{(u)},\widetilde\theta^{(u)},\widetilde\eta^{(u)}}(d\widetilde{\beta}^{(u)}) =1.
\]

We are ready to show that \(D\) is the exponential martingale associated with a certain \(\calF_u^\rho\)-martingale. By Ito's lemma, for \(u\geq 0\) we have
\begin{align*}
dD(u) &= \sum_{i\in V} \l( -(\widetilde{W}^{(u)} e^{\rho(u)})_i e^{\rho_i(u)} - \widetilde{\eta}_i^{(u)} e^{\rho_i(u)} -\frac{1}{2}\r) D(u) d\rho_i(u) \\
	&\qquad + \frac{1}{2} \sum_{i\in V} \Biggl( \l( -(\widetilde{W}^{(u)} e^{\rho(u)})_i e^{\rho_i(u)} - \widetilde{\eta}_i^{(u)} e^{\rho_i(u)} -\frac{1}{2}\r)^2  \\
	&\qquad \qquad +  \l( -(\widetilde{W}^{(u)} e^{\rho(u)})_i e^{\rho_i(u)} -\widetilde{W}^{(u)}_{i,i} e^{2\rho_i(u)} - \widetilde{\eta}_i^{(u)} e^{\rho_i(u)} \r) \Biggr) D(u) du \\
	& \qquad + \l(-\frac{1}{2}\langle e^{\rho(u)},\partial_u (\widetilde{W}^{(u)}) e^{\rho(u)}\rangle -\frac{1}{2}\langle \eta,\partial_u (H^{(u)})^{-1}\eta\rangle \r.  \\
	&\qquad\qquad \l. - \langle\partial_u\widetilde\eta^{(u)},e^{\rho(u)}\rangle  - \frac{|V|}{8} -\frac{1}{2}\frac{\partial_u|K^{(u)}|}{|K^{(u)}|} \r) D(u) du.
\end{align*}
Since \(H^{(u)}=2\beta^{(u)}-W=1/T(u)-W\), we have
\[
\partial_u (H^{(u)})^{-1}=(H^{(u)})^{-1}T(u)^{-1}\partial_u (T(u)) T(u)^{-1}(H^{(u)})^{-1},
\]
using Lemma \ref{lem:tech} (ii), we get
\begin{align*}
\langle\eta,\partial_u (H^{(u)})^{-1}\eta\rangle &= \langle T(u)^{-1}(H^{(u)})^{-1}\eta,e^{2\rho(u)}T(u)^{-1}(H^{(u)})^{-1}\eta\rangle \\
	&=\langle\widetilde\eta^{(u)},e^{2\rho(u)}\widetilde\eta^{(u)}\rangle = \sum_{i\in V}(\widetilde\eta_i^{(u)})^2 e^{2\rho_i(u)}.
\end{align*}
Moreover, \(\widetilde{W}^{(u)}=W+W (H^{(u)})^{-1} W\), therefore 
\begin{align*}
\langle e^{\rho(u)},\partial_u (\widetilde{W}^{(u)})e^{\rho(u)}\rangle &= \langle e^{\rho(u)},W (H^{(u)})^{-1}T(u)^{-1} e^{2\rho(u)} T(u)^{-1}(H^{(u)})^{-1} W e^{\rho(u)}\rangle \\
	&= \langle e^{\rho(u)}, \widetilde{W}^{(u)}e^{2\rho(u)}\widetilde{W}^{(u)} e^{\rho(u)}\rangle = \sum_{i\in V} (\widetilde{W}^{(u)}e^{\rho(u)})_i^2 e^{2\rho_i(u)},
\end{align*}
and \(\widetilde\eta^{(u)}=\widetilde{W}^{(u)}T(u)\eta +\eta\), thus
\begin{align*}
\langle\partial_u\widetilde\eta^{(u)},e^{\rho(u)}\rangle &= \langle \partial_u(\widetilde{W}^{(u)})T(u)\eta +\widetilde{W}^{(u)} \partial_u(T(u))\eta,e^{\rho(u)}\rangle \\
&= \langle \widetilde{W}^{(u)}e^{2\rho(u)}\widetilde{W}^{(u)}T(u)\eta +\widetilde{W}^{(u)} e^{2\rho(u)}\eta,e^{\rho(u)}\rangle\\
&= \langle \widetilde{W}^{(u)}e^{2\rho(u)}\widetilde\eta^{(u)},e^{\rho(u)}\rangle = \sum_{i\in V} \l(\widetilde{W}^{(u)}e^{\rho(u)}\r)_i\widetilde\eta_i^{(u)}e^{2\rho_i(u)}.
\end{align*}
Finally, we have
\begin{align*}
\partial_u |K^{(u)}| &= \Tr\big(|K^{(u)}|(K^{(u)})^{-1}\partial_u K^{(u)}\big)= -|K^{(u)}|\Tr\big(W (K^{(u)})^{-1} e^{2\rho(u)}\big)\\
 &= -|K^{(u)}| \sum_{i\in V}\widetilde{W}_{i,i}^{(u)}e^{2\rho_i(u)}.
\end{align*}

 Plug in the above computations in \(dD(u)\), we get that
\begin{align*}
\frac{dD(u)}{D(u)} &= \sum_{i\in V} \l( - \l(\widetilde{W}^{(u)}\l(e^{\rho(u)} +T(u)\eta\r) + \eta \r)_i e^{\rho_i(u)} -\frac{1}{2}\r) d\rho_i(u) \\
 & \;\; + \frac{1}{2}\sum_{i\in V} \l( (\widetilde{W}^{(u)}e^{\rho(u)})_i^2 e^{2\rho_i(u)} + (\widetilde\eta_i^{(u)})^2 e^{2\rho_i(u)} + \frac{1}{4}  + 2\l(\widetilde{W}^{(u)}e^{\rho(u)}\r)_i\widetilde\eta_i^{(u)}e^{2\rho_i(u)}\r. \\
 & \qquad + (\widetilde{W}^{(u)}e^{\rho(u)})_i e^{\rho_i(u)} + \widetilde{\eta}_i^{(u)}e^{\rho_i(u)}  -(\widetilde{W}^{(u)}e^{\rho(u)})_i e^{\rho_i(u)} -\widetilde{W}_{i,i}^{(u)}e^{2\rho_i(u)} - \widetilde{\eta}_i^{(u)}e^{\rho_i(u)} \\
 &\qquad \l. -(\widetilde{W}^{(u)}e^{\rho(u)})_i^2 e^{2\rho_i(u)} - (\widetilde\eta_i^{(u)})^2 e^{2\rho_i(u)} - 2\l(\widetilde{W}^{(u)}e^{\rho(u)}\r)_i\widetilde\eta_i^{(u)}e^{2\rho_i(u)} -\frac{1}{4} + \widetilde{W}_{i,i}^{(u)}e^{2\rho_i(u)} \r) du \\
 &= \sum_{i\in V} \l( - \l(\widetilde{W}^{(u)}\l(e^{\rho(u)} +T(u)\eta\r) + \eta \r)_i e^{\rho_i(u)} -\frac{1}{2}\r) d\rho_i(u)  = d \widetilde{L}(u),
\end{align*}
where for \(i\in V\) and \(u\geq 0\),
\[
\widetilde{L}_i(u) =\int_0^u \l(-\frac{1}{2} - \l(\widetilde{W}(u)\l(e^{\rho(u)} +T(u)\eta\r) + \eta \r)_i e^{\rho_i(u)}\r)d\rho_i(u).
\]
Therefore, \(D\) is the exponential martingale associated with \(\widetilde{L}\).

Recall that for \(u\geq 0\) and any event \(A_u\in\calF_u^\rho=\sigma\l(\rho(v),0\leq v\leq u\r)\), we have
\[
\P[A_u] = \int_{A_u} D(u) d\widehat{\P},
\]
\textit{i.e.} \(\P\) is such that
\[
\E\l[\frac{d\P}{d\widehat\P} \middle| \calF_u^\rho\r]=\calE(\widetilde{L})(u)
\]
for all \(u\geq 0\). Moreover, \(\widehat{\P}=\bigotimes_{i\in V}\widehat{\P}_i\), therefore \(\rho\) is a \(|V|\)-dimensional standard Brownian motion under \(\widehat\P\). By Girsanov's theorem, the process \(\widetilde{B}(u)=\rho(u)-\langle\rho,\widetilde{L}\rangle_u\) is a standard Brownian motion under \(\P\). In other words, under \(\P\), the process \(\rho\) verifies the following SDE : for all \(i\in V\) and \(u\geq 0\), 
\[
d\rho_i(u)=d\widetilde{B}_i(u)-\frac{1}{2}du - \l(\widetilde{W}(u)(e^{\rho(u)}+T(u)\eta)+\eta\r)_i e^{\rho_i(u)}du.
\]
\end{proof}
\subsection{Time change on the conditional process}
\begin{proof}[Proof of Theorem \ref{thm:md.opp.dr}]
Let \((\widetilde{B}_i)_{i\in V}\) be a \(|V|\)-dimensional standard Brownian motion. According to Theorem \ref{thm:SDE.rho}, there exists a \(|V|\)-dimensional standard Brownian motion \((B_i)_{i\in V}\) such that, if \((X_i)_{i\in V}\) is the solution of \eqref{SDE:X} with the Brownian motion \(B\), and \(T_i\) is the inverse function of \(U_i:t\mapsto\int_0^t\frac{ds}{X_i(s)^2}\) for all \(i\in V\), then \((\rho, T)\) is the solution of \eqref{SDE:rho} with the Brownian motion \(\widetilde{B}\), where \(\rho_i(u)=\log \big( X_i(T_i(u)) \big)\) for \(u\geq 0\). 

Therefore, according to Lemma \ref{lem:lim.u}, we have a.s. for all \(i\in V\) :
\[
\lim_{u\to +\infty}T_i(u)=\tau_i,
\]
where \(\tau_i\) is the hitting time of \(0\) by \(X_i\). In this coupling between \((\rho,T)\) and \(X\), it is natural to prefer the notation \(T^{\infty}=\tau\). Moreover, we can apply Theorem \ref{thm:SDE.X} (ii) to \(X\) : the vector \(\l(\frac{1}{2T^{\infty}_i}\r)_{i\in V}\) is distributed according to \(\nu_V^{W,\theta,\eta}\), and conditionally on \((T^{\infty}_i)_{i\in V}\), the trajectories \((X_i(t))_{0\leq t\leq T^{\infty}_i}\) are independent three-dimensional Bessel bridges from \(\theta_i\) to \(0\) respectively.

Since \(\rho_i(u)=\log\big(X_i(T_i(u))\big)\) for \(u\geq 0\), conditionally on \((T^{\infty}_i)_{i\in V}\), the processes \((\rho_i,T_i)_{i\in V}\) are independent, and their distribution is given by applying the time change from Lamperti's relation to a three-dimensional Bessel bridge. This time-change was already realized in the proof of Lemma \ref{lem:lim.u}, see Equation \eqref{tc.Bb}, and the result is as follows : conditionally on \((T^{\infty}_i)_{i\in V}\), for all \(i\in V\), there exists a standard Brownian motion \(B^*_i\) such that for \(u\geq 0\).
\[
\l\{\begin{aligned}
\rho_i(u)&=\log(\theta_i)+B^*_i(u)+\frac{1}{2}u+\log\l(\frac{T^{\infty}_i-T_i(u)}{T^{\infty}_i}\r)\\
T_i(u)&=\int_0^u e^{2\rho_i(w)}dw.
\end{aligned}\r.
\]
\end{proof}

\section{Conditioning in Lamperti time scale : Proof of Theorem \ref{thm:md.matyorcond}, \ref{thm:intertwinnings} and \ref{thm:equalities_in_law}}
\subsection{Proof of Theorem \ref{thm:md.matyorcond}}
In this section, we prove Theorem \ref{thm:md.matyorcond} by using the multidimensional opposite drift Theorem obtained in Theorem \ref{thm:md.opp.dr}. Besides, our proof follows the same computations as in the one dimensional case which is treated by Matsumoto and Yor in \cite{matyor2}.

\begin{proof}[Proof of Theorem \ref{thm:md.matyorcond}]
By \((ii)\) in Theorem \ref{thm:md.opp.dr}, there exists a \(|V|\)-dimensional Brownian motion \(B^*\) which is independent of \(T^{\infty}\) such that for every \(u\geq 0\),
\[\rho(u)=\log(\theta)+B^*(u)+\frac{1}{2}u+\log\left(\frac{T^{\infty}-T(u)}{T^{\infty}}\right).\]
For every \(i\in V\) and for every \(u\geq0\), let us define
\[e^*_i(u):=e^{B^*_i(u)+u/2}.\]
Therefore, for every \(i\in V\), for every \(u\geq 0\),
\begin{equation}
 \frac{T_i^{\infty}}{T_i^{\infty}-T_i(u)}\frac{e^{\rho_i(u)}}{\theta_i}=e^*_i(u).\label{matyor1}
 \end{equation}
Furthermore, for every \(i\in V\) and for every \(u\geq 0\), let us define \[T^*_i(u)=\int_0^ue_i^*(v)^2dv.\]
Integrating the square of identity \eqref{matyor1}, we have, for every \(i\in V\) and for every \(u\geq 0\),
\begin{align*}
\theta_i^2T_i^*(u)&=(T_i^{\infty})^2\int_0^u\frac{e^{2\rho_i(v)}}{(T_i^{\infty}-T_i(v))^2}dv\\
&=(T_i^{\infty})^2\left[\frac{1}{T_i^{\infty}-T_i(v)} \right]_0^u\\
&=(T_i^{\infty})^2\left(\frac{1}{T_i^{\infty}-T_i(u)} -\frac{1}{T_i^{\infty}}\right).
\end{align*}
Therefore, almost surely, for every \(i\in V\) and for every \(u\geq0\),
\begin{equation}\label{matyor2}
\theta_i^2T_i^*(u)=\frac{T_i^{\infty}T_i(u)}{T_i^{\infty}-T_i(u)}.
\end{equation}
For every \(i\in V\) and for every \(u\geq 0\), identity \eqref{matyor2} yields
\begin{equation}
\frac{1}{T_i^{\infty}}+\frac{1}{\theta_i^2T^*_i(u)}=\frac{1}{T_i(u)}.\label{matyor3}
\end{equation}
Let us differentiate \eqref{matyor3}. This gives that for every \(i\in V\) and for every \(u\geq 0\),

\begin{equation}
\frac{1}{\theta_i^2}\frac{e_i^*(u)^2}{T_i^*(u)^2}=\frac{e^{2\rho_i(u)}}{T_i(u)^2}.\label{matyor4}
\end{equation}
For every \(i\in V\) and for every \(u\geq 0\), we denote \(Z_i^*(u)=\frac{T_i^*(u)}{e_i^*(u)}\). Thus, by \eqref{matyor4}, for every \(i\in V\), almost surely
\begin{equation}
(\theta_i Z_i^*(u))_{u\geq 0}=(Z_i(u))_{u\geq 0}.\label{matyor5}
\end{equation}
The components of \(Z^*\) are independent. Moreover, by Theorem \ref{thm:matyorcond} with \(\mu=1/2\), for every \(i\in V\), \((Z_i^*(u))_{u\geq 0}\) is solution of 
\[dZ_i^*(u)=Z_i^*(u)d\widehat{B}_i(u)+\frac{K_{3/2}}{K_{1/2}}\left(\frac{1}{Z_i^*(u)}\right)du\]
for some brownian motion \(\widehat{B}\) which is different from \(B^*\).
However, it is not difficult to see that for every \(x\in\R\), \(\frac{K_{3/2}(x)}{K_{1/2}(x)}=1+\frac{1}{x}\). Therefore, for every \(i\in V\), \((Z_i^*(u))_{u\geq 0}\) is solution of the SDE:
\[dZ_i^*(u)=Z_i^*(u)d\widehat{B}_i(u)+(1+Z_i^*(u))du.\]

Together with  \eqref{matyor5}, this yields (ii) of Theorem \ref{thm:md.matyorcond}.
By \eqref{matyor3} and \eqref{matyor5}, we know that almost surely,
\begin{equation}
\left(Z_i(u),\frac{1}{T_i(v)}\right)_{u\geq 0,v\geq0,i\in V}=\left( \theta_iZ^*_i(u),\frac{1}{T_i^{\infty}}+\frac{1}{ \theta_i^2T_i^*(v)}\right)_{u\geq 0,v\geq 0,i\in V}.\label{matyor6}
\end{equation}
Remark that \(T^{\infty}\) is independent of \((Z^*(u))_{u\geq 0}\) because \((Z^*(u))_{u\geq 0}\) depends only on \(B^*\). Therefore, making \(v\) go to infinity in \eqref{matyor6}, we get that \((Z(u))_{u\geq 0}\) is independent of \(T^{\infty}\) which is (iii) in Theorem \ref{thm:md.matyorcond}.

Now, let \((\lambda_i)_{i\in V}\in\R_+^V\). Let \(z\in(\R_+^*)^V\). Recall that for every \(i\in V\) and for every \(u\geq 0 \), \(\beta_i(u)=\frac{1}{2T_i(u)}\). Let us look at the Laplace transform of \(\beta(u)\), conditionally on \(Z(u)=z\). By \eqref{matyor6}, we get
\begin{align*}
\E\left[\exp\left(-\sum\limits_{i\in V}\lambda_i\beta_i(u) \right)\Bigg|\mathcal{Z}_u,Z_u=z \right]=&\\
&\hspace{-4 cm}\E\left[\exp\left(-\sum\limits_{i\in V}\frac{\lambda_i}{2T_i^{\infty}} \right)\times \exp\left(-\sum\limits_{i\in V} \frac{\lambda_i}{2\theta_i^2T^*_i(u)}\right) \Bigg|\mathcal{Z}_u,Z(u)=z \right].
\end{align*}
By (i) in Theorem \ref{thm:md.opp.dr}, we know that the random vector \(1/(2T^{\infty})\) is distributed according to \(\nu_V^{W,\theta,\eta}\). Moreover, we know that \(T^{\infty}\) is independent of \(B^*\), that is, of \((Z^*,T^*)\). By \eqref{matyor5}, this implies that  \(T^{\infty}\) is independent of \((Z,T^*)\). Therefore, by Proposition \ref{prop:nuwtheta},
\begin{align*}
\E\left[\exp\left(-\sum\limits_{i\in V}\lambda_i\beta_i(u) \right)\Bigg|\mathcal{Z}_u,Z_u=z \right]=&\\
 &\hspace{-6.8 cm}e^{-\frac{1}{2}\langle\sqrt{\theta^2+\lambda},W\sqrt{\theta^2+\lambda}\rangle+\frac{1}{2}\langle\theta,W\theta\rangle+\langle\eta,\theta-\sqrt{\theta^2+\lambda}\rangle}\hspace{-0.1 cm}\times \hspace{-0.1 cm} \prod\limits_{i\in V}\frac{\theta_i}{\sqrt{\theta_i^2+\lambda_i}}\hspace{-0.1 cm}\times \hspace{-0.1 cm}\E\left[\exp\left(-\sum\limits_{i\in V} \frac{\lambda_i}{2\theta_i^2T^*_i(u)}\right) \Bigg|\mathcal{Z}_u,Z(u)=z \right].
\end{align*}
Besides, by \eqref{matyor5}, \(\theta Z^*=Z\). Therefore,
\begin{align}
\E\left[\exp\left(-\sum\limits_{i\in V}\lambda_i\beta_i(u) \right)\Bigg|\mathcal{Z}_u,Z_u=z \right]=&e^{-\frac{1}{2}\langle\sqrt{\theta^2+\lambda},W\sqrt{\theta^2+\lambda}\rangle+\frac{1}{2}\langle\theta,W\theta\rangle+\langle\eta,\theta-\sqrt{\theta^2+\lambda}\rangle}\times \prod\limits_{i\in V}\frac{\theta_i}{\sqrt{\theta_i^2+\lambda_i}}\nonumber\\ 
&\times \E\left[\exp\left(-\sum\limits_{i\in V}\frac{\lambda_i}{2\theta_i^2T^*_i(u)}\right) \Bigg|\mathcal{Z}_u,Z_i(^*u)=z_i /\theta_i\right].\label{matyor6bis}
\end{align}
Furthermore, we know that \((Z^*_i,T^*_i)_{i\in V}\) is site by site independent because for every \(i\in V\), \((Z_i^*,T_i^*)\) is a functional of \(B_i^*\). Consequently, we get that
\begin{align}
\E\left[\exp\left(-\sum\limits_{i\in V}\lambda_i\beta_i(u) \right)\Bigg|\mathcal{Z}_u,Z_u=z \right]&=e^{-\frac{1}{2}\langle\sqrt{\theta^2+\lambda},W\sqrt{\theta^2+\lambda}\rangle+\frac{1}{2}\langle\theta,W\theta\rangle+\langle\eta,\theta-\sqrt{\theta^2+\lambda}\rangle}\times \prod\limits_{i\in V}\frac{\theta_i}{\sqrt{\theta_i^2+\lambda_i}}\nonumber\\ 
&\times \prod\limits_{i\in V} \E\left[\exp\left(-\frac{\lambda_i}{2\theta_i^2T^*_i(u)}\right) \Bigg|\mathcal{Z}_u,Z_i(^*u)=z_i/\theta_i \right].\label{matyor7}
\end{align}
Now, for every \(i\in V\), Let's compute
\begin{align*}
\E\left[\exp\left(\frac{\lambda_i}{2\theta_i^2T^*_i(u)}\right) \Bigg|\mathcal{Z}_u,Z(u)=z \right] &=\E\left[\exp\left(- \frac{\lambda_i}{2\theta_i^2T^*_i(u)}\right)\Bigg|\mathcal{Z}_u,Z_i^*(u)=z_i/\theta_i \right].
\end{align*}
Recall that \(IG(\mu,r)\) designates an Inverse Gaussian distribution with parameter \((\mu,r)\).
By (iii) in Theorem \ref{thm:matyorcond}, the distribution of \(e^*_i(u)\) conditionally on \(Z_i^*(u)=z_i/\theta_i\) is
\[\mathcal{L}\left(e_i^*(u)|\mathcal{Z}_u,Z_i^*(u)=z_i/\theta_i \right)=\frac{1}{IG\left(1,\frac{\theta_i}{z_i}\right)}.\]
Moreover, for every \(t,\mu,r\in(0,+\infty)\), 
\[t\times IG(\mu,r)\overset{law}=IG(t\mu,tr).\]
Therefore, it holds that
\begin{align}
\mathcal{L}\left(\frac{1}{2\theta_i^2T_i^*(u)}\Bigg|\mathcal{Z}_u,Z_i^*(u)=z_i/\theta_i \right)=IG\left(\frac{1}{2\theta_iz_i},\frac{1}{2z_i^2}\right).\label{matyor7bis}
\end{align}
It is well known that the Laplace transform of an Inverse Gaussian random variable \(X\) with parameters \((\mu,r)\) is given by \[\E\left(e^{-tX} \right)=\exp\left(\frac{r}{\mu}\left(1-\sqrt{1+\frac{2\mu^2t}{r}} \right) \right).\] Consequently,
\begin{align*}
\E\left[\exp\left(-\frac{\lambda_i}{2\theta_i^2T^*_i(u)}\right) \Bigg|\mathcal{Z}_u,Z(u)=z \right] &=\exp\left(\frac{1}{z_i}\left(\theta_i-\sqrt{\theta_i^2+\lambda_i}\right)\right).
\end{align*}
Combining this with \eqref{matyor7} yields 
\begin{align}
\E\left[\exp\left(-\sum\limits_{i\in V}\lambda_i\beta_i(u) \right)\Bigg|\mathcal{Z}_u,Z_u=z \right]=&e^{-\frac{1}{2}\langle\sqrt{\theta^2+\lambda},W\sqrt{\theta^2+\lambda}\rangle+\frac{1}{2}\langle\theta,W\theta\rangle+\langle\eta+\frac{1}{z},\theta-\sqrt{\theta^2+\lambda}\rangle}\times \prod\limits_{i\in V}\frac{\theta_i}{\sqrt{\theta_i^2+\lambda_i}}.\nonumber
\end{align}
This is exactly the Laplace Transform of \(\nu_V^{W,\theta,\eta+1/z}\). This proves (iv) in Theorem \ref{thm:md.matyorcond}. Remark that (iv) implies directly (i). 
\end{proof}
\subsection{Proof of Theorem \ref{thm:intertwinnings}}
Let us prove the link between the solution \((\rho,T)\) of  \(E_V^{W,\theta,\eta}(\rho)\) and \(Z\) via intertwinnings.
\begin{proof}[Proof of Theorem \ref{thm:intertwinnings}]
Let \(v,u\in\R_+\) such that \(v<u\). Let \(f\) be a measurable function from \((\R^V)^2\) into \(\R_+\). 
On the one hand, it holds that,
\begin{align*}
\E\left[f(\rho(u),T(u))|\mathcal{Z}_v \right]&=\E\left[\E\left[f(\rho(u),T(u))|\sigma(\rho(w),T(w),w\leq v)\right]|\mathcal{Z}_v\right]\\
&=\E\left[P_{u-v}f(\rho(v),T(v))|\mathcal{Z}_v\right].
\end{align*}
Consequently, by (iv) in Theorem \ref{thm:md.matyorcond},
\begin{align*}
\E\left[f(\rho(u),T(u))|\mathcal{Z}_v \right]&= \E\left[P_{u-v}f\left( \Big(-\ln({2\beta_i(v)Z_i(v)}) \Big)_{i\in V},\left( \frac{1}{2\beta_i(v)}\right)_{i\in V} \right)\Bigg|\mathcal{Z}_v\right]\\
&=KP_{u-v}f(Z(v)).
\end{align*}
On the other hand, by (iv) in Theorem \ref{thm:md.matyorcond} again remark that,
\begin{align*}
\E\left[f(\rho(u),T(u))|\mathcal{Z}_v \right]&=\E\left[\E\left[f(\rho(u),T(u))|\mathcal{Z}_u\right]|\mathcal{Z}_v\right]\\
&=\E\left[Kf(Z(u))|\mathcal{Z}_v\right]\\
&=Q_{u-v}Kf(Z_v).
\end{align*}
Therefore, almost surely, \[Q_{u-v}Kf(Z(v))=KP_{u-v}f(Z(v)).\]

\end{proof}

\subsection{Proof of Theorem \ref{thm:equalities_in_law}}
In the proof of Theorem \ref{thm:equalities_in_law}, we use the same notation as in the beginning of the proof of Theorem \ref{thm:md.matyorcond}. For example, \(B^*\), \(e^*\), \(T^*\) and \(Z^*\) are defined in the same way as before.
\begin{proof}[Proof of Theorem \ref{thm:equalities_in_law}]
Recall that, by \eqref{matyor6}, 
\begin{equation}
\left(Z_i(u),\frac{1}{T_i(v)}\right)_{u\geq 0,v\geq0,i\in V}=\left( \theta_iZ^*_i(u),\frac{1}{T_i^{\infty}}+\frac{1}{ \theta_i^2T_i^*(v)}\right)_{u\geq 0,v\geq 0,i\in V}.\nonumber
\end{equation}
Let \(u\geq 0\). By \eqref{matyor6} and the fact that for every \(i\in V\), \(T_i(u)=Z_i(u)e^{\rho_i(u)}\) and \(T_i^*(u)=Z_i^*(u)e_i^*(u)\), we get
\begin{equation}
\frac{1}{e_i^*(u)}=\frac{\theta_i}{e^{\rho_i(u)}}-\frac{\theta_iZ_i(u)}{T_i^{\infty}}.
\label{matyor6bis}
\end{equation} 
Following \cite{matsumoto2003interpretation}, remark that by \eqref{matyor6bis} it holds that
\begin{align*}
\theta_i^2\frac{Z_i^*(u)}{e_i^*(u)}+\frac{T_i^{\infty}}{e_i^*(u)^2}&=\theta_i^2\frac{Z_i(u)}{\theta_i}\left( \frac{\theta_i}{e^{\rho_i(u)}}-\frac{\theta_iZ_i(u)}{T_i^{\infty}}\right)+T_i^{\infty}\left(\frac{\theta_i}{e^{\rho_i(u)}}-\frac{\theta_iZ_i(u)}{T_i^{\infty}} \right)^2\\
&=\theta_i^2\left(\frac{T_i^{\infty}}{e^{2\rho_i(u)}} -\frac{Z_i(u)}{e^{\rho_i(u)}}\right).
\end{align*}
However, as \(T_i(u)=Z_i(u)e^{\rho_i(u)}\), we get
\begin{align}
\theta_i^2\frac{Z_i^*(u)}{e_i^*(u)}+\frac{T_i^{\infty}}{e_i^*(u)^2}&=\theta_i^2 \frac{T_i^{\infty}-T_i(u)}{e^{2\rho_i(u)}}\nonumber\\
&=\theta_i^2\int_0^{+\infty}e^{2\rho_i(u+v)-2\rho_i(u)}dv.\label{matyor8}
\end{align}
Therefore, combining \eqref{matyor6} and \eqref{matyor8} yields
\begin{equation}
\left(\frac{1}{T_i(u)},\int_0^{+\infty}e^{2\rho_i(u+v)-2\rho_i(u)}dv\right)_{i\in V}=\left( \frac{1}{T_i^{\infty}}+\frac{1}{ \theta_i^2T_i^*(u)},\frac{Z_i^*(u)}{e_i^*(u)}+\frac{T_i^{\infty}}{\theta_i^2e_i^*(u)^2}\right)_{i\in V}.\label{matyor9}
\end{equation}
Now, the idea is to condition both sides of \eqref{matyor9} on \(Z(u)=z\) and to use (iv) in Theorem \ref{thm:md.matyorcond}. Let us begin with the left-hand side. This term can be rewritten as
\begin{equation}
\left(\frac{1}{T_i(u)},\frac{Z_i(u)^2}{T_i(u)^2}\int_0^{+\infty}e^{2\rho_i(u+v)}dv\right)_{i\in V}.\label{matyor10}
\end{equation}
First, let us condition on \(\sigma(\rho(v),v\leq u)\). By Theorem \ref{thm:Markov.X} in the exponential scale, conditionally on \(\sigma(\rho(v),v\leq u)\), \((\rho(u+v))_{v\geq 0}\) is distributed as a solution of \(E_V^{\widetilde{W}(u),\widetilde{\theta}(u),\widetilde{\eta}(u)}\) where \(\widetilde{W}(u)=W K_{T(u)}^{-1}\), \(\widetilde{\eta}(u)=\eta+W K_{T(u)}^{-1}T(u)\eta\) and for every \(i\in V\), \(\widetilde{\theta}_i(u)=e^{\rho_i(u)}=1/(2\beta_i(u)Z_i(u))\). Thus, by (i) in Theorem \ref{thm:md.opp.dr}, conditionally on \(\sigma(\rho(v),v\leq u)\), 
\[ \left( \left(2\int_0^{+\infty}e^{2\rho_i(u+v)}dv\right)^{-1}\right)_{i\in V} \]
is distributed as \(\alpha\sim\nu_V^{\widetilde{W}(u),\widetilde{\theta}(u),\widetilde{\eta}(u)}\).
Recall that, by \((iv)\) in Theorem \ref{thm:md.matyorcond}, conditionally on \(\mathcal{Z}_u\), \(1/(2T(u))\) is distributed as \(\beta\sim\nu_V^{W,\theta,\eta}\). Therefore, if we condition \eqref{matyor10}, first on \(\sigma(\rho(v),v\leq u)\) and then on \(Z(u)=z\), we obtain 
\[\left(\left(2\beta_i\right)_{i\in V}, \left(\frac{z_i^2(2\beta_i)^2}{2\alpha_i}\right)_{i\in V}\right)\]
where \(\beta\sim \nu_V^{W,\theta,\eta+1/z}\) and conditionally on \(\beta\), \(\alpha\sim\nu_V^{\widetilde{W},\widetilde{\theta},\widetilde{\eta}}\).

Now, let us look at the right-hand side in \eqref{matyor9}. This right-hand side can be rewritten as
\[\left( \frac{1}{T_i^{\infty}}+\frac{1}{ \theta_i^2T_i^*(u)},\frac{Z_i^*(u)^2}{T_i^*(u)}+\frac{T_i^{\infty}Z_i^*(u)^2}{\theta_i^2T_i^*(u)^2}\right)_{i\in V}.\]
We know that \(T^{\infty}\) is independent of \(B^*\), thus of  \((Z^*,T^*)\). By \eqref{matyor5}, this implies that \(T^{\infty}\) is independent of \((Z,T^*)\). Moreover, by (i) in Theorem \ref{thm:md.opp.dr}, \(\delta=1/(2T^{\infty})\) is distributed as \(\nu_V^{W,\theta,\eta}\). Thus, conditionally on \((Z_i(u)=z_i)_{i\in V}\), that is, \((Z_i^*(u)=z_i/\theta_i)_{i\in V}\), we have
\begin{align}
\left( \frac{1}{T_i^{\infty}}+\frac{1}{ \theta_i^2T_i^*(u)},\frac{Z_i^*(u)^2}{T_i^*(u)}+\frac{T_i^{\infty}Z_i^*(u)^2}{\theta_i^2T_i^*(u)^2}\right)_{i\in V}&\overset{law}=\left(2\delta_i+\frac{1}{ \theta_i^2T_i^*(u)},\frac{z_i^2}{\theta_i^2T_i^*(u)} +\frac{z_i^2}{2\delta_i(\theta_i^2T_i^*(u))^2}\right)_{i\in V}
\end{align}
where \((T_i^*(u))_{i\in V}\) is independent of \((\delta_i)_{i\in V}\). 
Moreover, by \eqref{matyor7bis}, conditionally on \((Z_i^*(u)=z_i/\theta_i)_{i\in V}\), \((T^*_i(u))_{i\in V}\) are independent site by site and for every \(i\in V\), \(1/(\theta_i^2T_i^*(u))\) is distributed as \[IG(1/(\theta_iz_i),1/z_i^2).\]
\end{proof}

\end{document}